\documentclass[a4paper,11pt]{article}
\usepackage{latexsym,amssymb,amsmath,amsthm,amsxtra,xypic}
\usepackage{stmaryrd}
\usepackage{amsfonts}
\usepackage{bbm}
\usepackage{amscd}
\usepackage{mathrsfs}
\usepackage[dvips]{graphicx}
\usepackage[all]{xy}

\newtheorem{thm}{Theorem}[section]

\newtheorem{exa}[thm]{Example}

\newtheorem{defi}[thm]{Definition}

\newcommand{\beq}{\begin{equation}}
\newcommand{\eeq}{\end{equation}}

\def\DD{D}

\newcommand{\trl}{\triangleleft}
\newcommand{\trr}{\triangleright}

\newcommand{\vphi}{\varphi}

\newcommand {\emptycomment}[1]{} %to remove paragraphs

\newcommand{\be }{\begin{equation}}
\newcommand{\ee }{\end{equation}}

\newcommand{\pf}{\noindent{\bf Proof.}\ }

\newcommand{\g}{\mathfrak g}
\newcommand{\h}{\mathfrak h}
\newcommand{\m}{\mathfrak m}

\newcommand{\huaV}{\mathcal{V}}

\newcommand{\br}[1]{   [ \cdot,    \cdot  ]   }

\newcommand{\dM}{\mathrm{d}}

\newcommand{\gl}{\mathfrak {gl}}

\newcommand{\Ker}{\mathrm{Ker}}

\newcommand{\End}{\mathrm{End}}
\newcommand{\ad}{\mathrm{ad}}

\newcommand{\ve}{\mathrm{v}}

\newcommand{\V}{\mathbb{V}}

\allowdisplaybreaks

\begin{document}

\title{{Lie triple 2-algebras}}%Categorification of Lie Triple Systems  } }

\author{Tao Zhang, Zhang-Ju Liu}
%\thanks {This research is supported by NSFC(11961049). }}

%\author{Tao Zhang\\
%Department of Mathematics, Henan Normal University,\\
%Xinxiang 453007,  China
%\\
%Email:zhangtao@htu.cn}

\date{}


%\footnotetext{{This research is supported by NSFC(11961049). }}

\maketitle

\begin{abstract}
We construct a new cohomology theory for Lie triple algebras.
Using this cohomology, we introduce the notions of  2-term $L_\infty$-triple algebras and Lie triple 2-algebras. We prove that the category of 2-term  $L_\infty$-triple algebras is equivalent to the category of Lie triple 2-algebras.
Crossed modules of  Lie triple algebras are studied in detail.
\end{abstract}

%\tableofcontents

\section{Introduction}

The notion of Lie triple algebras (also called a general Lie triple system or a Lie-Yamaguti algebra) was introduced by Yamaguti in \cite{Yam0}. Yamaguti provided the representation and cohomology theory for Lie triple algebras in \cite{Yam1}. Kinyon and Weinstein investigated its relationship with Courant algebroids and enveloping Lie algebra in \cite{KW}. Lie-Yamaguti algebras related to simple Lie algebras and irreducible Lie-Yamaguti algebras were studied in \cite{BDE, BEM}. The formal deformation and extension theory for Lie triple algebras were investigated in \cite{LCM, Zhang1}. However, it has been found that the integrable obstruction of formal deformation can not be controlled by the Yamaguti's cohomology group.
Recently, the free connection algebras relate to Lie triple algebras were studied in \cite{MS,St}.

The algebraic theory of Lie 2-algebras, which is a categorification of Lie algebras, was extensively studied by Baez and Crans in \cite{BC}.
This is a 2-vector space endowed with a natural transformation known as the Jacobiator, which must satisfy certain coherence laws.
It is well know that this  coherence laws are given by the 3-cocycle condition for Lie algebras.
Recently, categorification of Leibniz algebras, omni-Lie algebras, Lie superalgebras and 3-Lie algebras was given in \cite{SL,SLZ, ZL, ZLS}.
For further advancements in this area of algebraic structures, refer to \cite{LST}.

The aim of this paper is to solve the following two problems.
The first problem is whether there exists a new cohomology theory for Lie triple algebras that goes beyond Yamaguti's cohomology. The second problem is whether it is possible to give a categorification of Lie triple algebras. We provide positive answers to both of these questions.
Roughly speaking, Yamaguti's cochain complex was defined as follows:
\begin{eqnarray*}
 C(T,V)\xrightarrow[]{\delta} C^{(2,3)}(T,V)\xrightarrow[]{\delta} C^{(4,5)}(T,V)\xrightarrow[]{\delta} C^{(6,7)}(T,V)\xrightarrow[]{\delta}\cdots
\end{eqnarray*}
In this paper, we extend Yamaguti's cochain complex by constructing a new cochain complex:
\begin{eqnarray*}
 C(T,V)\xrightarrow[]{\Delta_1}  C^{(2,3)}(T,V)\xrightarrow[]{\Delta_2}C^{(3,4)}(T,V)\times C^{(4,5)}(T,V)
\xrightarrow[]{\Delta_3} C^{(5,6)}(T,V)\times C^{(6,7)}(T,V)\xrightarrow[]{\Delta_5}\cdots
\end{eqnarray*}
We will verify that $\Delta_3\circ\Delta_2=0$, which allows us to obtain a new $(3,4,4,5)$-cohomology theory. Further details of this cochain complex can be found in Theorem \ref{mainthm1}.
This new cohomology theory plays a crucial role in the categorification of Lie triple algebras. We introduce two key notions derived from a cohomological and categorical point of view, namely 2-term $L_\infty$-triple algebras and Lie triple 2-algebras.
Due to the complicity nature of the cohomology of Lie triple algebras, it is difficult to obtain 2-term $L_\infty$-triple algebras using directly Yamaguti's cohomology for Lie triple algebras.
We overcome this difficulty by using this new $(3,4,4,5)$-cohomology theory for Lie triple algebras.
We establish the equivalence between the category of 2-term $L_\infty$-triple algebras and the category of Lie triple 2-algebras. For a detailed proof, refer to Theorem \ref{mainthm2}.

The second part of this paper is devoted to investigating some special cases of Lie triple 2-algebras.
The first case is the {\em skeletal} Lie triple 2-algebra which is equivalent to a Lie triple algebra with a $(3,4,4,5)$-cocycle.
The second one is the {\em strict} Lie triple 2-algebra which is equivalent to the crossed module of Lie triple algebras.
We will construct crossed module of Lie triple algebras by using the crossed module of Leibniz algebras and the reductive crossed module of Lie algebras.
Finally, we will investigate the crossed module extensions of Lie triple algebras by using the $(3,4,4,5)$-cohomology groups.

 The paper is organized as follows. In Section 2, we revisit the concept of Lie triple algebras and study its cohomology. In Section 3, we define the concept of  2-term $L_\infty$-triple algebras and show that there is a category of 2-term $L_\infty$-triple algebras. In Section 4, we introduce the notion of Lie triple 2-algebras, and prove that the category of 2-term  $L_\infty$-triple algebras is equivalent to the category of Lie triple 2-algebras. In Section 5, we first introduce the concept of crossed modules of Lie triple algebras. Then we show that there is a one-to-one correspondence between strict Lie triple 2-algebras and crossed modules of Lie triple algebras. Finally, we classify   crossed module extensions of Lie triple algebras by the $(3,4,4,5)$-cohomology groups.

Throughout this paper, all algebras are assumed to be over an algebraically closed field $k$ of characteristic different from 2 and 3.
The space of linear maps from a vector space $V$ to $V$ is denoted by $\End(V)$.

\section{A new cohomology theory for Lie triple algebras}
In this section, we  first review definitions and notations related to Lie triple algebras.

%{Lie triple algebras and their cohomology}

\begin{defi}[\cite{Yam0}]\label{def:LYA}
A Lie triple algebra consists of a vector space $T$ together with a bilinear map and a trilinear
map $[\cdot,\cdot]:  T\times T\to  T$, $[\cdot,\cdot,\cdot]:  T\times T\times T\to  T$ such that: $\forall\, x_i, y_i \in  T$,
\begin{itemize}
\item[$\bullet$] {\rm(LY1)}\quad $[x_1,x_2]+[x_2,x_1]=0$;
\item[$\bullet$] {\rm(LY2)}\quad $[x_1, x_2,x_3]+[x_2, x_1,x_3]=0$;
\item[$\bullet$] {\rm(LY3)}\quad $[[x_1, x_2],x_3]+c.p.+[x_1, x_2,x_3]+c.p.=0$;
%\item[$\bullet$] {\rm(LY3')}\quad $[[x_1, x_2],x_3]+c.p.=0,\quad[x_1, x_2,x_3]+c.p.=0$;
\item[$\bullet$] {\rm(LY4)}\quad $[[x_1, x_2],x_3,y_1]+c.p.=0$;
\item[$\bullet$] {\rm(LY5)}\quad $[x_1, x_2, [y_1, y_2]] = [[ x_1, x_2, y_1], y_2] + [y_1, [ x_1, x_2, y_2]]$;
\item[$\bullet$] {\rm(LY6)}\quad
$[x_1, x_2, [y_1, y_2, y_3]] = [[ x_1, x_2, y_1], y_2, y_3] + [y_1, [ x_1, x_2, y_2], y_3] + [y_1, y_2, [ x_1, x_2, y_3]]$,
\end{itemize}
for all $x_i, y_i \in T$, where c.p. means cyclic permutations with respect to $x_1$, $x_2$ and $x_3$.
\end{defi}
Note that (LY1) and (LY2) means the bilinear map $[\cdot,\cdot]$ is the antisymmetric and the trilinear map $[\cdot,\cdot,\cdot]$ is antisymmetric only in the first two variables. We call (LY3) the Jacobi identity and  (LY6) the fundamental identity. Any Lie algebra considered with the trilinear map $[x_1, x_2,x_3]:=[[x_1, x_2],x_3]$ is a Lie triple algebra. In this case, (LY3) is just the  Jacobi identity of Lie algebra.
%On the other hand, if $[\cdot,\cdot]\equiv0$, then (LY2), (LY3) and (LY6) define a Lie triple system.
Another special case of Lie triple algebra with $[\cdot,\cdot]=0$  is a Lie triple system which only satisfying then (LY2), (LY3) and (LY6) .
The cohomology theory of  Lie triple systems was studied  in \cite{KT,Zhang0}

A homomorphism between two Lie triple algebras $T$ and $T'$ is a map $\varphi:T\to T'$ satisfying
\begin{eqnarray}
\varphi([x_1, x_2])=[\varphi (x_1), \varphi (x_2)]',\quad \varphi([x_1, x_2,x_3])=[\varphi (x_1), \varphi (x_2),\varphi (x_3)]'.
\end{eqnarray}
%$$[x_1, x_2]=\varphi\inv[\varphi x_1, \varphi x_2]',\quad [x_1, x_2,x_3]=\varphi\inv[\varphi (x_1), \varphi (x_2),\varphi (x_3)]'.$$

Denoted by $L:=\otimes{}^2{T}$, which is called fundamental set.
The elements $X=(x_1,x_2), Y=(y_1,y_2)\in \otimes^2 {T}$ are called fundamental objects.
Define an operation on fundamental objects by
\begin{eqnarray}\label{eq:fundamental}
X\circ Y=([x_1,x_2,y_1],y_2)+(y_1,[x_1,x_2,y_2]).
\end{eqnarray}
It is easy to prove that $L$ is a Leibniz algebra.  By condition (LY6),  we define $\ad^L(x_1,x_2)(w)=[x_1,x_2,w]$, then one get the following equality
\begin{eqnarray}
\ad^L(X)\ad^L(Y) (w)-\ad^L(Y)\ad^L(X) (w)=\ad^L(X\circ Y) (w), %\quad [\ad(x),\ad(y)]=\ad(x\circ y)$$
\end{eqnarray}
for all $X,Y\in L, w\in {T}$. Thus $\ad^L: L \to \gl({T})$ is a homomorphism of Leibniz algebras.

\begin{defi}[\cite{Yam1}] Let $T$ be a Lie triple algebra and $V$ be a vector space. Then $(V, \rho, D, \theta)$ is called a representation
of $T$ if and only if the following conditions are satisfied,
\begin{itemize}
\item[$\bullet$]{\rm(R31)}\quad $\DD(x_1,x_2)-\theta(x_2,x_1)+\theta(x_1,x_2)+\rho([x_1,x_2])-[\rho (x_1),\rho (x_2)]=0$;
\item[$\bullet$]{\rm(R41)}\quad $\DD([x_1,x_2],x_3)+\DD([x_2,x_3],x_1)+\DD([x_3,x_1],x_2)=0$;
%\item[$\bullet$]{\rm(R42)}\quad $\DD([x_1,x_2],x_3)=\DD(x_1,x_3)\rho(x_2)-\rho(x_1)\DD(x_2,x_3)$;
\item[$\bullet$]{\rm(R42)}\quad $\theta([x_1,x_2],x_3)=\theta(x_1,x_3)\rho(x_2)-\theta(x_2,x_3)\rho(x_1)$;
%\item[$\bullet$]{\rm(R51)}\quad $\DD(x_1,x_2)\rho(y_1)=\rho([x_1,x_2,y_1])+\rho(y_1)\DD(x_1,x_2)$;
\item[$\bullet$]{\rm(R51)}\quad $[\DD(x_1,x_2),\rho(y_2)]=\rho([x_1, x_2, y_2])$;
\item[$\bullet$]{\rm(R52)}\quad $\theta(x_1,[y_1, y_2])=\rho(y_1)\theta(x_1, y_2)-\rho(y_2)\theta(x_1,y_1)$;
\item[$\bullet$]{\rm(R61)}\quad $[\DD(x_1,x_2),\theta(y_1,y_2)]=\theta((x_1,x_2)\circ (y_1,y_2))$;
%\item[$\bullet$]{\rm(R3)}\quad $[\DD(x_1,x_2),\DD(y_1,y_2)]=\DD((x_1,x_2)\circ (y_1,y_2))$,
\item[$\bullet$]{\rm(R62)}\quad $\theta(x_1,[y_1, y_2, y_3])= \theta (y_2, y_3)\theta(x_1,y_1) - \theta (y_1, y_3)\theta(x_1,y_2) + \DD (y_1, y_2)\theta(x_1,y_3)$,
\end{itemize}
where $\rho$ is a map from $T$ to $\End(V)$ and $D, \theta$ are maps from $T\times T$ to $\End(V)$.
\end{defi}

%\begin{itemize}
%\item[$\bullet$]{\rm(R31)}\quad $[x_1,x_2,u_3]+[x_2,u_3,x_1]+[u_3,x_1,x_2]+[[x_1,x_2],u_3]+[[x_2,u_3],x_1]+[[u_3,x_1],x_2]=0$;
%\item[$\bullet$]{\rm(R41)}\quad $[[x_1,x_2],x_3,u_1]+[[x_2,x_3],x_1,u_1]+[[x_3,x_1],x_2,u_1]=0$;
%\item[$\bullet$]{\rm(R42)}\quad $[u_1,[x_1,x_2],x_3]=[[x_2,u_1],x_1,x_3]-[[x_1,u_1],x_2,x_3]$;
%\item[$\bullet$]{\rm(R51)}\quad $[x_1,x_2,[y_2,u_1]]-[y_2,[x_1,x_2,u_1]]=[[x_1, x_2, y_2],u_1],$;
%\item[$\bullet$]{\rm(R52)}\quad $[u_2,x_1,[y_1, y_2]]=[y_1, [u_2,x_1, y_2]]-[y_2,[u_2,x_1,y_1]]$;
%\item[$\bullet$]{\rm(R61)}\quad $[\DD(x_1,x_2),\theta(y_1,y_2)]=\theta((x_1,x_2)\circ (y_1,y_2))$;
%\item[$\bullet$]{\rm(R62)}\quad $\theta(x_1,[y_1, y_2, y_3])= \theta (y_2, y_3)\theta(x_1,y_1) - \theta (y_1, y_3)\theta(x_1,y_2) + \DD (y_1, y_2)\theta(x_1,y_3)$,
%\end{itemize}

For example, given a Lie triple algebra $T$, there is a natural {\bf adjoint representation} on itself.
The corresponding representation maps $D(x_1,x_2)$, $\theta(x_1,x_2)$, $\rho(x_1)$  are given by
\begin{eqnarray*}
D(x_1,x_2)(x_3)=[x_1,x_2,x_3],\quad \theta(x_1,x_2)(x_3)=[x_3,x_1,x_2],\quad \rho(x_1)(x_2)=[x_1,x_2].
\end{eqnarray*}

Next we revisit  the Yamaguti's cohomology for Lie triple algebras.

Let $V$ be a representation of Lie triple algebra $T$. We are going to define the cohomology groups of $T$ with coefficients in $V$.
Let $f:T\times\cdots \times T\to V$ be multilinear maps of $T$ into $V$ such that the following conditions are satisfied:
\begin{eqnarray}
 f(x_1,\cdots,x_{2i-1},x_{2i}\cdots,x_n)&=&0,\ \  \mbox{if}\ \  x_{2i-1}=x_{2i}, \forall i=1,2,\cdots,\lfloor n/2 \rfloor.
\end{eqnarray}
The vector space spanned by such multilinear maps is called an cochain of $T$, which is denoted by $C^n(T,V)$ for $n\geq 1$.

\begin{defi}[\cite{Yam1}]\label{def:cohomology}
Let $n\geq 1$, $(f,g)\in C^{2n}(T, V)\times C^{2n+1}(T, V)$.
The coboundary operator $\delta: (\widetilde{f}, g)\mapsto (\delta_{\textrm{I}}\widetilde{f}, \delta_{\textrm{II}}g)$
is a mapping from $C^{2n}(T, V)\times C^{2n+1}(T, V)$ into $C^{2n+2}(T, V)\times C^{2n+3}(T, V)$ defined as follows:
\begin{eqnarray*}
 &&(\delta_{\textrm{I}}\widetilde{f})(x_{1}, x_{2}, \cdots, x_{2n+2})\\
&=&\rho(x_{2n+1})g(x_{1}, \cdots, x_{2n}, x_{2n+2}))-\rho(x_{2n+2})g(x_{1}, \cdots, x_{2n+1})\\
&&-g(x_{1}, \cdots, x_{2n}, [x_{2n+1},x_{2n+2}])\\
&&+\sum\limits_{k=1}^{n}(-1)^{n+k+1}D(x_{2k-1}, x_{2k})\widetilde{f}(x_{1}, \cdots, \hat{x}_{2k-1}, \hat{x}_{2k}, \cdots, x_{2n+2})\\
&&+\sum\limits_{k=1}^{n}\sum\limits_{j=2k+1}^{2n+2}(-1)^{n+k}\widetilde{f}(x_{1}, \cdots, \hat{x}_{2k-1}, \hat{x}_{2k}, \cdots, [x_{2k-1}, x_{2k}, x_{j}], \cdots, x_{2n+2}),
\end{eqnarray*}
\begin{eqnarray*}
 &&(\delta_{\textrm{II}}g)(x_{1}, x_{2}, \cdots, x_{2n+3})\\
&=& \theta(x_{2n+2},  x_{2n+3})g(x_{1},\cdots, x_{2n+1})\\
&&-\theta( x_{2n+1},  x_{2n+3})g(x_{1}, \cdots, x_{2n}, x_{2n+2})\\
&&+\sum\limits_{k=1}^{n+1}(-1)^{n+k+1}D( x_{2k-1},  x_{2k})g(x_{1}, \cdots, \hat{x}_{2k-1}, \hat{x}_{2k}, \cdots, x_{2n+3})\\
&&+\sum\limits_{k=1}^{n+1}\sum\limits_{j=2k+1}^{2n+3}(-1)^{n+k}g(x_{1}, \cdots, \hat{x}_{2k-1}, \hat{x}_{2k}, \cdots, [x_{2k-1}, x_{2k}, x_{j}], \cdots,x_{2n+3}).
\end{eqnarray*}
\end{defi}

The coboundary operator defined above satisfies
$\delta\circ \delta=0$, that is, $\delta_{\textrm{I}} \circ\delta_{\textrm{I}} =0$ and $\delta_{\textrm{II}} \circ\delta_{\textrm{II}} =0.$
Let $Z^{(2n,2n+1)}(T, V)$ be the subspace of $C^{2n}(T, V)\times C^{2n+1}(T, V)$ spanned by $(\widetilde{f}, g)$ such that $\delta(\widetilde{f}, g)=0$
which is called the space of cocycles and $B^{(2n,2n+1)}(T, V)=\delta(C^{2n-2}(T, V)\times C^{2n-1}(T, V))$ which is called the space of coboundaries.

\begin{defi}[\cite{Yam1}]\label{def:1cohmologygroup} For the case $n\geq 2$,
the $(2n, 2n+1)$-cohomology group  of a Lie triple algebra $T$ with coefficients in $V$ is defined to be the quotient space:
$$H^{(2n, 2n+1)}(T, V)\triangleq(Z^{(2n,2n+1)}(T, V))/(B^{(2n,2n+1)}(T, V)).$$
\end{defi}

For $n=1$, the $(2,3)$-cohomology group is defined as follows.
Let $C^2(T,V)$ be the space of maps $\nu:  T\times T\to V$ such that
\begin{eqnarray}\label{eq:skewsymm01}
\nu(x_1,x_2)+\nu(x_2,x_1)=0,
\end{eqnarray}
and $C^3(T,V)$ the space of maps $\omega:  T\times T\times T\to V$ such that
\begin{eqnarray}\label{eq:skewsymm02}
\omega(x_1, x_2, x_3)+\omega(x_2, x_1, x_3)=0.
\end{eqnarray}
The coboundary operator
$$\delta: C^2(T,V)\times C^3(T,V)\to C^4(T,V)\times C^5(T,V)\quad (\nu,\omega)\mapsto(\delta_{\textrm{I}}\nu, \delta_{\textrm{II}}\omega)$$
is defined by
\begin{eqnarray}
\delta_{\textrm{I}}\nu(x_1,x_2, y_1, y_2)&=&\omega( x_1, x_2,[y_1, y_2])+\rho(y_2)\omega(x_1,x_2,y_1)\notag\\
&&-\rho(y_1)\omega(x_1,x_2,y_2)+\DD(x_1, x_2)\nu(y_1, y_2)\notag\\
&&\label{3coboundary01}-\nu([x_1, x_2, y_1], y_2) - \nu(y_1,[x_1, x_2,y_2])),\\
\delta_{\textrm{II}}\omega(x_1,x_2, y_1, y_2, y_3)&=&
\omega( x_1, x_2,[y_1, y_2, y_3])-\omega([x_1, x_2, y_1], y_2, y_3)\notag\\
&& - \omega(y_1, [x_1, x_2,y_2], y_3) - \omega(y_1, y_2, [x_1, x_2,y_3])\notag\\
&&+\DD(x_1, x_2)\omega(y_1, y_2, y_3)-\theta(y_2, y_3)\omega(x_1,x_2,y_1)\notag\\
&&\label{3coboundary02}+\theta(y_1, y_3)\omega(x_1,x_2,y_2)- \DD(y_1, y_2)\omega(x_1,x_2,y_3).
\end{eqnarray}

The coboundary operator
$$\delta^*: C^2(T,V)\times C^3(T,V)\to C^3(T,V)\times C^4(T,V),\quad (\nu,\omega)\mapsto(\delta^*_{\textrm{I}}\nu, \delta^*_{\textrm{II}}\omega)$$
is defined by
  \begin{eqnarray}
\delta^*_{\textrm{I}}\nu(x_1,x_2, x_3)&=&\omega(x_1,x_2, x_3)+c.p.+\rho(x_1)\nu(x_2,x_3)+c.p.\notag\\
&&\label{3coboundary03}+\nu([x_1,x_2],x_3)+c.p.,\\
\label{3coboundary04}\delta^*_{\textrm{II}}\omega(x_1,x_2, x_3,y_1)
&=&\theta(x_1, y_1)\nu(x_2,x_3)+c.p.+\omega([x_1,x_2], x_3, y_1)+c.p.
\end{eqnarray}

Thus we obtain a cobourdary operator
\begin{eqnarray*}
\Delta_2=(\delta^*_{\textrm{I}},\delta^*_{\textrm{II}},\delta_{\textrm{I}},\delta_{\textrm{II}}):
&&C^{(2,3)}(T,V)\longrightarrow C^{(3,4)}(T,V)\times C^{(4,5)}(T,V),\\
 &&(\nu,\omega)\mapsto(\delta^*_{\textrm{I}}\nu, \delta^*_{\textrm{II}}\omega, \delta_{\textrm{I}}\nu, \delta_{\textrm{II}}\omega),
\end{eqnarray*}
where we denote $C^{(3,4)}(T,V):= C^3(T,V)\times C^4(T,V)$ and $C^{(4,5)}(T,V):=C^4(T,V)\times C^5(T,V)$.

The (2,3)-cohomology group  of a Lie triple algebra $T$ with coefficients in $V$ is defined to be the quotient space
$$H^{(2,3)}(T, V))\triangleq Z^{(2,3)}(T, V)/B^{(2,3)}(T, V).$$

Finally, we construct a new $(3,4,4,5)$-cohomlogy theory for Lie triple algebras. We extend the Yamaguti's cohomology spaces to the following:
\begin{eqnarray*}
\Delta_3=(\delta^*_{\textrm{I}},\delta^*_{\textrm{II}},\delta_{\textrm{I}},\delta_{\textrm{II}}):
 &&C^{(3,4)}(T,V)\times C^{(4,5)}(T,V)\longrightarrow C^{(5,6)}(T,V)\times C^{(6,7)}(T,V),\\
&&(l_3,\widehat{l}_4, \widetilde{l}_4,l_5)\mapsto (\delta^*_{\textrm{I}}(l_3), \delta^*_{\textrm{II}}(\widehat{l}_4), \delta_{\textrm{I}}(\widetilde{l}_4), \delta_{\textrm{II}}(l_5)).
\end{eqnarray*}
The cobourdary operator  $\Delta_3=(\delta^*_{\textrm{I}},\delta^*_{\textrm{II}},\delta_{\textrm{I}},\delta_{\textrm{II}})$ is defined  as follows:
\begin{align}\label{3cocycle01}
{\delta^*_{\textrm{I}}(l_3)(x_1,x_2,y_1,y_2,y_3)}=&D(x_1,x_2)l_3(y_1,y_2,y_3)-l_3([x_1,x_2,y_1],y_2,y_3)\notag\\
&-l_3(y_1,[x_1,x_2,y_2],y_3)-l_3(y_1,y_2,[x_1,x_2,y_3])\notag\\
&+\widetilde{l}_4(x_1,x_2,[y_1,y_2],y_3)+\widetilde{l}_4(x_1,x_2,y_2,[y_1,y_3])\notag\\
&-\widetilde{l}_4(x_1,x_2,y_1,[y_2,y_3])-\rho(y_1)\widetilde{l}_4(x_1,x_2,y_2,y_3)\notag\\
&+\rho(y_2)\widetilde{l}_4(x_1,x_2,y_1,y_3)-\rho(y_3)\widetilde{l}_4(x_1,x_2,y_1,y_2)\notag\\
&-l_5(x_1,x_2,y_1,y_2,y_3)-l_5(x_1,x_2,y_2,y_3,y_1)\notag\\
&-l_5(x_1,x_2,y_3,y_1,y_2),
\end{align}
\begin{align}\label{3cocycle02}
{\delta^*_{\textrm{II}}(\widehat{l}_4)(x_1,x_2,y_1,y_2,y_3,z_1)}=
&D(x_1,x_2)\widehat{l}_4(y_1,y_2,y_3,z_1)\notag\\
&-\widehat{l}_4([X,y_1],y_2,y_3,z_1)-\widehat{l}_4(y_1,[X,y_2],y_3,z_1)\notag\\
&-\widehat{l}_4(y_1,y_2,[X,y_3],z_1)-\widehat{l}_4(y_1,y_2,y_3,[X,z_1])\notag\\
&+\theta(y_1,z_1)\widetilde{l}_4(x_1,x_2,y_2,y_3)+\theta(y_3,z_1)\widetilde{l}_4(x_1,x_2,y_1,y_2)\notag\\
&-\theta(y_2,z_1)\widetilde{l}_4(x_1,x_2,y_1,y_3)+l_5(X,[y_1,y_2],y_3,z_1)\notag\\
&+l_5(X,y_2,[y_1,y_3],z_1)-l_5(X,y_1,[y_2,y_3],z_1),
\end{align}
%{\footnotesize
\begin{align}\label{3cocycle03}
{\delta_{\textrm{I}}(\widetilde{l}_4)(x_1,x_2,y_1,y_2,z_1,z_2)}=
&D(x_1,x_2)\widetilde{l}_4(Y,z_1,z_2)-D(y_1,y_2)\widetilde{l}_4(X,z_1,z_2)\notag\\
&+\rho(z_1)l_5(X,Y,z_2)-\rho(z_2)l_5(X,Y,z_1)-l_5(X,Y,[z_1,z_2])\notag\\
&-\widetilde{l}_4([X,y_1],y_2,z_1,z_2)-\widetilde{l}_4(y_1, [X,y_2],z_1,z_2)\notag\\
&-\widetilde{l}_4(Y,[X,z_1],z_2)-\widetilde{l}_4(Y,z_1,[X,z_2])\notag\\
&+\widetilde{l}_4(X, [Y,z_1],z_2)+\widetilde{l}_4(X,z_1,[Y,z_2]),
\end{align}
%}
%{\footnotesize
\begin{align}\label{3cocycle04}
{\delta_{\textrm{II}}(l_5)(x_1,x_2,y_1,y_2,z_1,z_2,z_3)}=
&\theta(z_2,z_3)l_5(X,Y,z_1)+\theta(z_1,z_3)l_5(X,Y,z_2)\notag\\
&+D(x_1,x_2)l_5(Y,z_1,z_2,z_3)-D(y_1,y_2)l_5(X,z_1,z_2,z_3)\notag\\
&+D(z_1,z_2)l_5(X,Y,z_3)-l_5([X,y_1],y_2,z_1,z_2,z_3)\notag\\
&-l_5(y_1,[X,y_2],z_1,z_2,z_3)-l_5(Y,[X,z_1],z_2,z_3)\notag\\
&-l_5(Y,z_1,[X,z_2],z_3)-l_5(Y,z_1,z_2,[X,z_3])\notag\\
&+l_5(X,[Y,z_1],z_2,z_3)+l_5(X,z_1,[Y,z_2],z_3)\notag\\
&+l_5(X,z_1,z_2,[Y,z_3])-l_5(X,Y,[z_1,z_2,z_3]),
\end{align}
%}
where $X=(x_1,x_2)$, $Y=(y_1,y_2)$, $l_3\in C^{3}(T,V),\widehat{l}_4\in C^{4}(T,V), \widetilde{l}_4\in C^{4}(T,V)$ and $l_5\in C^{5}(T,V)$.
Note that the above cobourdary operators  $\delta_{\textrm{I}}(\widetilde{l}_4)$ and $\delta_{\textrm{II}}(l_5)$ are Yamaguti's cohomology for the case $n=2$, but the cobourdary operators  $\delta^*_{\textrm{I}}(l_3)$ and $\delta^*_{\textrm{II}}(\widehat{l}_4)$ are absolutely new.

Thus we obtain the following cochain complex
\begin{eqnarray*}
 C^{(2,3)}(T,V)&\xrightarrow[]{\Delta_2=(\delta^*,\delta)}&C^{(3,4)}(T,V)\times C^{(4,5)}(T,V)
\xrightarrow[]{\Delta_3=(\delta^*,\delta)} C^{(5,6)}(T,V)\times C^{(6,7)}(T,V).
\end{eqnarray*}
The cohomology group in the middle space of  this cochain complex are to be defined as follows.
%\begin{eqnarray*}
% C^2(T,V)\times C^3(T,V)&\longrightarrow& C^3(T,V)\times C^4(T,V)\times C^3(T,V)\times C^4(T,V)\\
% &\longrightarrow& C^5(T,V)\times C^6(T,V)\times C^6(T,V)\times C^7(T,V)
%\end{eqnarray*}

\begin{defi}
Let $T$ be a Lie triple algebra and $(V, \rho, D, \theta)$ a representation of $T$.
Then $(l_3,\widehat{l}_4, \widetilde{l}_4,l_5)\in C^{(3,4,4,5)}(T,V)$ is called a (3,4,4,5)-cocycle if it is contained in the kernel of $\Delta_3$, i.e.,
${\Delta_3}(l_3,\widehat{l}_4, \widetilde{l}_4,l_5)=(\delta^*_{\textrm{I}}(l_3), \delta^*_{\textrm{II}}(\widehat{l}_4), \delta_{\textrm{I}}(\widetilde{l}_4), \delta_{\textrm{II}}(l_5))=0$.
The space of (3,4,4,5)-cocycles is denoted by $Z^{(3,4,4,5)}(T, V)$.

For $(l_3,\widehat{l}_4, \widetilde{l}_4,l_5)\in C^{(3,4,4,5)}(T,V)$, it is called a (3,4,4,5)-coboundary if there exists  $(\nu,\omega)\in C^2(T,V)\times C^3(T,V)$  such that  $(l_3,\widehat{l}_4, \widetilde{l}_4,l_5)={\Delta_2}(\nu,\omega)=(\delta^*_{\textrm{I}}\nu, \delta^*_{\textrm{II}}\omega, \delta_{\textrm{I}}\nu, \delta_{\textrm{II}}\omega)$.
The space of (3,4,4,5)-coboundaries is denoted by $B^{(3,4,4,5)}(T, V)$.
\end{defi}

\begin{thm}\label{mainthm1} With notations above, we have
$$\Delta_3\circ \Delta_2=0.$$
Thus the space of (3,4,4,5)-coboundaries is contained in space of (3,4,4,5)-cocycles.
\end{thm}
\begin{proof}
Assume $l_3=\delta^*_{\textrm{I}}\nu, \widehat{l}_4=\delta^*_{\textrm{II}}\omega, \widetilde{l}_4=\delta_{\textrm{I}}\nu, l_5=\delta_{\textrm{II}}\omega$ as defined in \eqref{3coboundary01}--\eqref{3coboundary04}. By Yamaguti's cohomology,
$\delta_{\textrm{I}}(\widetilde{l}_4)=\delta_{\textrm{I}}(\delta_{\textrm{I}}\nu)=0, \delta_{\textrm{II}}(l_5)=\delta_{\textrm{II}}(\delta_{\textrm{II}}\omega)=0$.
Now we verify that $\delta^*_{\textrm{I}}(\delta^*_{\textrm{I}}\nu)=0$ and  $\delta^*_{\textrm{II}}(\delta^*_{\textrm{II}}\omega)=0$.
By definition, we have
\begin{eqnarray*}
&&\delta^*_{\textrm{I}}(\delta^*_{\textrm{I}}\nu)(x_1,x_2,y_1,y_2,y_3)\\
%&=&\delta^*_{\textrm{I}}(l_3)(x_1,x_2,y_1,y_2,y_3)\\
%&=&D(x_1,x_2)l_3(y_1,y_2,y_3)-l_3([x_1,x_2,y_1],y_2,y_3) \\
%&&-l_3(y_1,[x_1,x_2,y_2],y_3)-l_3(y_1,y_2,[x_1,x_2,y_3]) \\
%&&+\widetilde{l}_4(x_1,x_2,[y_1,y_2],y_3)+\widetilde{l}_4(x_1,x_2,y_2,[y_1,y_3]) \\
%&&-\widetilde{l}_4(x_1,x_2,y_1,[y_2,y_3])-\rho(y_1)\widetilde{l}_4(x_1,x_2,y_2,y_3) \\
%&&+\rho(y_2)\widetilde{l}_4(x_1,x_2,y_1,y_3)-\rho(y_3)\widetilde{l}_4(x_1,x_2,y_1,y_2)\\
%&&-l_5(x_1,x_2,y_1,y_2,y_3)-l_5(x_1,x_2,y_2,y_3,y_1)\\
%&&-l_5(x_1,x_2,y_3,y_1,y_2)\\
&=&D(x_1,x_2)\Big(\omega(y_1,y_2, y_3)+c.p.+\rho(y_1)\nu(y_2,y_3)+c.p.+\nu([y_1,y_2],y_3)+c.p.\Big)\\
&&-\Big(\omega([x_1,x_2,y_1],y_2, y_3)+c.p.+\rho([x_1,x_2,y_1])\nu(y_2,y_3)+c.p.\\
&&+\nu([[x_1,x_2,y_1],y_2],y_3)+c.p.\Big)\\
&&-\Big(\omega(y_1,[x_1,x_2,y_2], y_3)+c.p.+\rho(y_1)\nu([x_1,x_2,y_2],y_3)+c.p.\\
&&+\nu([y_1,[x_1,x_2,y_2]],y_3)+c.p.\Big)\\
&&-\Big(\omega(y_1,y_2, [x_1,x_2,y_3])+c.p.+\rho(y_1)\nu(y_2,[x_1,x_2,y_3])+c.p.\\
&&+\nu([y_1,y_2],[x_1,x_2,y_3])+c.p.\Big)\\
&&+\Big(\omega( x_1, x_2,[[y_1,y_2], y_3])+\rho(y_3)\omega(x_1,x_2,[y_1,y_2])-\rho([y_1,y_2])\omega(x_1,x_2,y_3)\\
&&\quad+\DD(x_1, x_2)\nu([y_1,y_2], y_3)-\nu([x_1, x_2, [y_1,y_2]], y_3) - \nu([y_1,y_2],[x_1, x_2,y_3])\Big)\\
&&+\Big(\omega( x_1, x_2,[y_2, [y_1,y_3]])+\rho([y_1,y_3])\omega(x_1,x_2,y_2)-\rho(y_2)\omega(x_1,x_2,[y_1,y_3])\\
&&\quad+\DD(x_1, x_2)\nu(y_2, [y_1,y_3])-\nu([x_1, x_2, y_2], [y_1,y_3]) - \nu(y_2,[x_1, x_2,[y_1,y_3]])\Big)\\
&&-\Big(\omega( x_1, x_2,[y_1, [y_2,y_3]])+\rho([y_2,y_3])\omega(x_1,x_2,y_1)-\rho(y_1)\omega(x_1,x_2,[y_2,y_3])\\
&&\quad+\DD(x_1, x_2)\nu(y_1, [y_2,y_3])-\nu([x_1, x_2, y_1], [y_2,y_3]) - \nu(y_1,[x_1, x_2,[y_2,y_3]])\Big)\\
&&-\rho(y_1)\Big(\omega( x_1, x_2,[y_2, y_3])+\rho(y_3)\omega(x_1,x_2,y_2)-\rho(y_2)\omega(x_1,x_2,y_3)\\
&&\quad+\DD(x_1, x_2)\nu(y_2, y_3)-\nu([x_1, x_2, y_2], y_3) - \nu(y_2,[x_1, x_2,y_3])\Big)\\
&&+\rho(y_2)\Big(\omega( x_1, x_2,[y_1, y_3])+\rho(y_3)\omega(x_1,x_2,y_1)-\rho(y_1)\omega(x_1,x_2,y_3)\\
&&\quad+\DD(x_1, x_2)\nu(y_1, y_3)-\nu([x_1, x_2, y_1], y_3) - \nu(y_1,[x_1, x_2,y_3])\Big)\\
&&-\rho(y_3)\Big(\omega( x_1, x_2,[y_1, y_2])+\rho(y_2)\omega(x_1,x_2,y_1)-\rho(y_1)\omega(x_1,x_2,y_2)\\
&&\quad+\DD(x_1, x_2)\nu(y_1, y_2)-\nu([x_1, x_2, y_1], y_2) - \nu(y_1,[x_1, x_2,y_2])\Big)\\
&=&\Big(\omega( x_1, x_2,[y_1, y_2, y_3])-\omega([x_1, x_2, y_1], y_2, y_3)\notag\\
&& - \omega(y_1, [x_1, x_2,y_2], y_3) - \omega(y_1, y_2, [x_1, x_2,y_3])\notag\\
&&+\DD(x_1, x_2)\omega(y_1, y_2, y_3)-\theta(y_2, y_3)\omega(x_1,x_2,y_1)\notag\\
&&\label{3coboundary02}+\theta(y_1, y_3)\omega(x_1,x_2,y_2)- \DD(y_1, y_2)\omega(x_1,x_2,y_3)\Big)+c.p.\\
&&-l_5(x_1,x_2,y_1,y_2,y_3)-c.p.\\
&=&0,
\end{eqnarray*}
where in the second equality we use conditions (R31) and (R51).
Similarly, we get
\begin{eqnarray*}
&&\delta^*_{\textrm{II}}(\delta^*_{\textrm{II}}\omega)(x_1,x_2,y_1,y_2,y_3,z_1)\\
%&=&\delta^*_{\textrm{II}}(\widehat{l}_4)(x_1,x_2,y_1,y_2,y_3,z_1) \\
%&=&D(x_1,x_2)\widehat{l}_4(y_1,y_2,y_3,z_1) \\
%&&-\widehat{l}_4([X,y_1],y_2,y_3,z_1)-\widehat{l}_4(y_1,[X,y_2],y_3,z_1) \\
%&&-\widehat{l}_4(y_1,y_2,[X,y_3],z_1)-\widehat{l}_4(y_1,y_2,y_3,[X,z_1]) \\
%&&+\theta(y_1,z_1)\widetilde{l}_4(x_1,x_2,y_2,y_3)+\theta(y_3,z_1)\widetilde{l}_4(x_1,x_2,y_1,y_2) \\
%&&-\theta(y_2,z_1)\widetilde{l}_4(x_1,x_2,y_1,y_3)\\
&=&D(x_1,x_2)\Big(\theta(y_1, z_1)\nu(y_2,y_3)+c.p.+\omega([y_1,y_2], y_3, z_1)+c.p.\Big)\\
&&-\Big(\theta([X,y_1], z_1)\nu(y_2,y_3)+c.p.+\omega([[X,y_1],y_2], y_3, z_1)+c.p.\Big)\\
&&-\Big(\theta(y_1, z_1)\nu([X,y_2],y_3)+c.p.+\omega([y_1,[X,y_2]], y_3, z_1)+c.p.\Big)\\
&&-\Big(\theta(y_1, z_1)\nu(y_2,[X,y_3])+c.p.+\omega([y_1,y_2], [X,y_3], z_1)+c.p.\Big)\\
&&-\Big(\theta(y_1, [X,z_1])\nu(y_2,y_3)+c.p.+\omega([y_1,y_2], y_3, [X,z_1])+c.p.\Big)\\
&&+\theta(y_1,z_1)\Big(\omega( x_1, x_2,[y_2, y_3])+\rho(y_3)\omega(x_1,x_2,y_2)-\rho(y_2)\omega(x_1,x_2,y_3)\\
&&\quad+\DD(x_1, x_2)\nu(y_2, y_3)-\nu([x_1, x_2, y_2], y_3) - \nu(y_2,[x_1, x_2,y_3])\Big)\\
&&+\theta(y_3,z_1)\Big(\omega( x_1, x_2,[y_1, y_2])+\rho(y_2)\omega(x_1,x_2,y_1)-\rho(y_1)\omega(x_1,x_2,y_2)\\
&&\quad+\DD(x_1, x_2)\nu(y_1, y_2)-\nu([x_1, x_2, y_1], y_2) - \nu(y_1,[x_1, x_2,y_2])\Big)\\
&&-\theta(y_2,z_1)\Big(\omega( x_1, x_2,[y_1, y_3])+\rho(y_3)\omega(x_1,x_2,y_1)-\rho(y_1)\omega(x_1,x_2,y_3)\\
&&\quad+\DD(x_1, x_2)\nu(y_1, y_3)-\nu([x_1, x_2, y_1], y_3) - \nu(y_1,[x_1, x_2,y_3])\Big)\\
&&-l_5(x_1,x_2,[y_1,y_2],y_3,z_1)-c.p.\\
&=&\Big(\omega( x_1, x_2,[[y_1,y_2], y_3, z_1])-\omega([x_1, x_2, [y_1,y_2]], y_3, z_1)\notag\\
&& - \omega([y_1,y_2], [x_1, x_2,y_3], z_1) - \omega([y_1,y_2], y_3, [x_1, x_2,z_1])\notag\\
&&+\DD(x_1, x_2)\omega([y_1,y_2], y_3, z_1)-\theta(y_3, z_1)\omega(x_1,x_2,[y_1,y_2])\notag\\
&&\label{3coboundary02}+\theta([y_1,y_2], z_1)\omega(x_1,x_2,y_3)- \DD([y_1,y_2], y_3)\omega(x_1,x_2,z_1)\Big)+c.p.\\
&&-l_5(x_1,x_2,[y_1,y_2],y_3,z_1)-c.p.\\
&=&0,
\end{eqnarray*}
where in the second equality we use conditions (R41),  (R42) and (R61).
Then we obtain $\delta^*_{\textrm{I}}(l_3)=\delta^*_{\textrm{I}}(\delta^*_{\textrm{I}}\nu)=0$ and  $\delta^*_{\textrm{II}}(\widehat{l}_4)=\delta^*_{\textrm{II}}(\delta^*_{\textrm{II}}\omega)=0$.
Thus the spaces of (3,4,4,5)-coboundaries  are contained in space of (3,4,4,5)-cocycles.
\end{proof}

\begin{defi}\label{def:1cohmologygroup}
The (3,4,4,5)-cohomology group  of a Lie triple algebra $T$ with coefficients in $V$ is defined to be the quotient space
$$H^{(3,4,4,5)}(T, V))\triangleq Z^{(3,4,4,5)}(T, V)/B^{(3,4,4,5)}(T, V).$$
\end{defi}

\section{$2$-term $L_\infty$-triple algebras}\label{sec:2term}

In this section, we introduce the notion of  2-term $L_\infty$-triple algebras.

\begin{defi}\label{lem:2term 3L}
A $2$-term $L_\infty$-triple algebra $\huaV=(V_1,V_0,\dM,l_3,\widehat{l}_4, \widetilde{l}_4,l_5),$ consists of the following data:
\begin{itemize}
\item[$\bullet$] a complex of vector spaces $V_1\stackrel{\dM}{\longrightarrow}V_0,$

\item[$\bullet$] a bilinear maps $[\cdot,\cdot]:V_i\times V_j\longrightarrow
V_{i+j}$, where $0\leq i+j\leq 1$,

\item[$\bullet$] a trilinear maps $[\cdot,\cdot,\cdot]:V_i\times V_j\times V_k\longrightarrow
V_{i+j+k}$, where $0\leq i+j+k\leq 1$,

\item[$\bullet$] a trilinear map $l_3:V_0\times V_0\times V_0\longrightarrow V_1$,

\item[$\bullet$] a pair of multilinear map $\widetilde{l}_4, \widehat{l}_4:V_0\times V_0\times V_0\times V_0\longrightarrow
V_1$,

\item[$\bullet$] a  multilinear map $l_5:V_0\times V_0\times V_0\times V_0\times V_0\longrightarrow
V_1$,
   \end{itemize}
   such that for any $x,y,x_i,y_i\in V_0$ and $u,v,u_i\in V_1$, the following equalities are satisfied:
\begin{itemize}
\item[$\rm(a)$] $ [x_1,x_2]=-[x_2,x_1],\quad [x, u]=- [u, x]$,
\item[] $ [x_1,x_2, x_3]=-[x_2,x_1,x_3],\quad  [x_1,x_2,u]=- [x_2,x_1,u],\quad  [u, x_1,x_2]=- [x_1,u,x_2]$,
\item[$\rm(b)$] $\dM [x, u]=[x,\dM u],\quad [\dM u,v]=[u,\dM v]$,
%\quad $[u_1,u_2]=0,$
\item[$\rm(c)$] $\dM [x_1,x_2,u]=[x_1,x_2,\dM u],$\quad $\dM [u,x_1,x_2]=[\dM u,x_1,x_2],$
%\item[$\rm(c)$] $[u_1,u_2,x_1]=0, \quad [x_1,u_1,u_2]=0,\quad  [u_1,u_2,u_3]=0,$
\item[$\rm(d)$] $[\dM u,v,x]=[u,\dM v,x]$,\quad $[x,\dM u,v]=[x,u,\dM v],$
\item[$\rm(e_1)$]
\begin{eqnarray*}
\dM l_3(x_1,x_2,x_3)&=&[[x_1, x_2],x_3]+c.p.+[x_1, x_2,x_3]+c.p.,\\
 l_3(x_1,x_2,\dM u_3)&=&[[x_1, x_2],u_3]+c.p.+[x_1, x_2,u_3]+c.p.,
\end{eqnarray*}
\item[$\rm(e_2)$]
\begin{eqnarray*}
\dM \widehat{l}_4(x_1, x_2,x_3,y_1)&=&[[x_1, x_2],x_3,y_1]+c.p.,\\
 \widehat{l}_4(x_1, x_2,x_3,\dM u_1)&=&[[x_1, x_2],x_3,u_1]+c.p.,
\end{eqnarray*}
\item[$\rm(e_3)$]
\begin{eqnarray*}
\dM \widetilde{l}_4(x_1, x_2, y_1, y_2)&=& [[ x_1, x_2, y_1], y_2] + [y_1, [ x_1, x_2, y_2]]-[x_1, x_2, [y_1, y_2]], \\
 \widetilde{l}_4(x_1, x_2, y_1, \dM u_2)&=&  [[ x_1, x_2, y_1], u_2] + [y_1, [ x_1, x_2, u_2]]-[x_1, x_2, [y_1, u_2]],\\
 \widetilde{l}_4(x_1,\dM u_2, y_1, y_2)&=& [[ x_1, u_2, y_1], y_2] + [y_1, [ x_1, u_2, y_2]]-[x_1, u_2, [y_1, y_2]] ,
\end{eqnarray*}
\item[$\rm(e_4)$]
\begin{eqnarray*}
\dM l_5(x_1, x_2,y_1,y_2,y_3)&=&[[x_1, x_2,y_1],y_2,y_3]+[y_1,[x_1, x_2,y_2],y_3]\\
&&+[y_1,y_2,[x_1, x_2,y_3]]-[x_1, x_2,[y_1,y_2,y_3]],\\
 l_5(x_1, x_2,y_1,y_2,\dM u_3)&=&[[x_1, x_2,y_1],y_2,u_3]+[y_1,[x_1, x_2,y_2],u_3]\\
&&+[y_1,y_2,[x_1, x_2,u_3]]-[x_1, x_2,[y_1,y_2,u_3]],\\
l_5(x_1,\dM u_2,y_1,y_2,y_3)&=&+[[x_1,u_2,y_1],y_2,y_3]+[y_1,[x_1,u_2,y_2],y_3]\\
&&+[y_1,y_2,[x_1,u_2,y_3]]-[x_1,u_2,[y_1,y_2,y_3]],
\end{eqnarray*}
\item[$\rm(f_1)$]
\begin{eqnarray}
&&[x_1,x_2,l_3(y_1,y_2,y_3)]+\widetilde{l}_4(x_1,x_2,[y_1,y_2],y_3)+\widetilde{l}_4(x_1,x_2,y_2,[y_1,y_3])\notag\\
&&+[\widetilde{l}_4(x_1,x_2,y_1,y_2),y_3])+[y_2,\widetilde{l}_4(x_1,x_2,y_1,y_3)])\notag\\
&&+l_5(x_1,x_2,y_1,y_2,y_3)+l_5(x_1,x_2,y_3,y_1,y_2)\notag\\
&=&\widetilde{l}_4(x_1,x_2,y_1,[y_2,y_3])+[y_1,\widetilde{l}_4(x_1,x_2,y_2,y_3)])\notag\\
&&+l_3([x_1,x_2,y_1],y_2,y_3)+l_3(y_1,[x_1,x_2,y_2],y_3)\notag\\
&&+l_3(y_1,y_2,[x_1,x_2,y_3])+l_5(x_1,x_2,y_3,y_2,y_1),
\end{eqnarray}
\item[$\rm(f_2)$]
\begin{eqnarray}
&&[x_1,x_2,\widehat{l}_4(y_1,y_2,y_3,z_1)]+l_5(X,[y_1,y_2],y_3,z_1)+l_5(X,y_2,[y_1,y_3],z_1)\notag\\
&&[\widetilde{l}_4(x_1,x_2,y_1,y_2),y_3,z_1]+[y_2,\widetilde{l}_4(x_1,x_2,y_1,y_3),z_1]\notag\\
&=&l_5(X,y_1,[y_2,y_3],z_1)+\widehat{l}_4([X,y_1],y_2,y_3,z_1)+[y_1,\widetilde{l}_4(x_1,x_2,y_2,y_3),z_1]\notag\\
&&+\widehat{l}_4(y_1,y_2,y_3,[X,z_1])+\widehat{l}_4(y_1,[X,y_2],y_3,z_1)+\widehat{l}_4(y_1,y_2,[X,y_3],z_1),
\end{eqnarray}
\item[$\rm(f_3)$]
%{\footnotesize
\begin{eqnarray}
&&[X,\widetilde{l}_4(Y,z_1,z_2)]+\widetilde{l}_4(X, [Y,z_1],z_2)+\widetilde{l}_4(X,z_1,[Y,z_2])\notag\\
&&+[l_5(X,Y,z_1),z_2]+ [z_1,l_5(X,Y,z_2)]\notag\\
&=& l_5(X,Y,[z_1,z_2])+[Y,\widetilde{l}_4(X,z_1,z_2)]+\widetilde{l}_4(X\circ Y,z_1,z_2)\notag\\
&&+\widetilde{l}_4(Y,[X,z_1],z_2)+\widetilde{l}_4(Y,z_1,[X,z_2]),
\end{eqnarray}
%}
\item[$\rm(f_4)$]
%{\footnotesize
\begin{eqnarray}
&&[l_5(X,Y,z_1),z_2,z_3]+[z_1,l_5(X,Y,z_2),z_3]+[X,l_5(Y,z_1,z_2,z_3)]\notag\\
&&[z_1,z_2,l_5(X,Y,z_3)]+l_5(X,[Y,z_1],z_2,z_3)\notag\\
&&+l_5(X,z_1,[Y,z_2],z_3)+l_5(X,z_1,z_2,[Y,z_3])\notag\\
&=&[Y,l_5(X,z_1,z_2,z_3)]+l_5([X,y_1],y_2,z_1,z_2,z_3)+l_5(y_1,[X,y_2],z_1,z_2,z_3)\notag\\
&&+l_5(Y,[X,z_1],z_2,z_3)+l_5(Y,z_1,[X,z_2),z_3)\notag\\
&&+l_5(X,Y,[z_1,z_2,z_3])+l_5(Y,z_1,z_2,[X,z_3]),
\end{eqnarray}
%}
where $X=(x_1,x_2)$ and $Y=(y_1,y_2)$.
   \end{itemize}
\end{defi}

\begin{defi}\label{homomorphism}
Let $\huaV=(V_1,V_0,\dM,l_3,\widehat{l}_4, \widetilde{l}_4,l_5)$ and $\huaV'=(V_1',V_0',\dM',l'_3,\widehat{l}'_4, \widetilde{l}'_4,l'_5)$ be two $2$-term $L_\infty$-triple algebras. A  homomorphism $\phi:\huaV \longrightarrow \huaV'$ consists of:
\begin{itemize}
\item[$\bullet$] a chain map $\phi:\huaV \longrightarrow \huaV'$, which consists of linear maps $\phi_0:V_0 \longrightarrow V_0'$ and $\phi_1:V_1 \longrightarrow V_1'$ such that $\phi_0\circ\dM=\dM' \circ\phi_1$;
\item[$\bullet$] two multilinear maps $\phi_2:V_0 \times V_0\longrightarrow V_1'$, $\phi_3:V_0 \times V_0 \times V_0 \longrightarrow V_1'$,
\end{itemize}
such that  for all $x_i\in V_0$ and $h\in V_1,$ we have
 \begin{eqnarray}
 \label{eq:homo01}\dM' (\phi_2(x_1,x_2))&=&\phi_0([x_1,x_2])-[\phi_0(x_1),\phi_0(x_2)]',\\
 \label{eq:homo02} \phi_2(x_1,\dM u)&=&\phi_1([x_1,u])-[\phi_0(x_1),\phi_1(u)]',\\
 \label{eq:homo03}\dM' (\phi_3(x_1,x_2,x_3))&=&\phi_0([x_1,x_2,x_3])-[\phi_0(x_1),\phi_0(x_2),\phi_0(x_3)]',\\
 \label{eq:homo04} \phi_3(x_1,x_2,\dM u)&=&\phi_1([x_1,x_2,u])-[\phi_0(x_1),\phi_0(x_2),\phi_1(u)]',
 \end{eqnarray}
       \begin{eqnarray}
\nonumber&&[\phi_2(x_1,x_2),\phi_0(x_3)]'+\phi_2([x_1,x_2],x_3)+\phi_1(l_3(x_1,x_2,x_3))\\
\nonumber&=&l_3^\prime(\phi_0(x_1),\phi_0(x_2),\phi_0(x_2))+[\phi_0(x_1),\phi_2(x_2,x_3)]'+[\phi_2(x_1,x_3),\phi_0(x_2)]'\\
\label{eq:morph11}&&+\phi_2(x_1,[x_2,x_3])+\phi_2([x_1,x_3],x_2),
      \end{eqnarray}
      and
    \begin{eqnarray}
    \nonumber&&[\phi_0(x_1),\phi_0(x_2),\phi_3(x_3,x_4,x_5)]'+\phi_3(x_1,x_2,[x_3,x_4,x_5])+\phi_1(l_5(x_1,x_2,x_3,x_4,x_5))\\
     \nonumber &=& l_5'(\phi_0(x_1),\phi_0(x_2),\phi_0(x_3),\phi_0(x_4),\phi_0(x_5))+l_3'(\phi_3(x_1,x_2,x_3),\phi_0(x_4),\phi_0(x_5))\\
    \nonumber &&+[\phi_0(x_3),\phi_3(x_1,x_2,x_4),\phi_0(x_5)]'+[\phi_0(x_3),\phi_0(x_4),\phi_3(x_1,x_2,x_5)]'\\
\label{eq:morph12} &&+\phi_3([x_1,x_2,x_3],x_4,x_5)+\phi_3(x_3,[x_1,x_2,x_4],x_5)+\phi_3(x_3,x_4,[x_1,x_2,x_5]).
     \end{eqnarray}
\end{defi}

Let $\varphi:\huaV \longrightarrow \huaV'$ and $\psi:\huaV' \longrightarrow \huaV''$ be two homomorphisms, their {\bf composition } $((\varphi\circ\psi)_0,(\varphi\circ\psi)_1,(\varphi\circ\psi)_2),(\varphi\circ\psi)_3)$ is given by $(\varphi\circ\psi)_0=\varphi_0\circ\psi_0$,  $(\varphi\circ\psi)_1=\varphi_1\circ\psi_1$,  $(\varphi\circ\psi)_2(x_1,x_2)=\psi_3(\varphi_0(x_1),\varphi_0(x_2))+\psi_1(\varphi_2(x_1,x_2))$ and
 $$(\varphi\circ\psi)_3(x_1,x_2,x_3)=\psi_3(\varphi_0(x_1),\varphi_0(x_2),\varphi_0(x_3))+\psi_1(\varphi_3(x_1,x_2,x_3)).$$
  The {identity homomorphism} $1_{\huaV}:\huaV\longrightarrow\huaV$ has the identity chain map as its underlying map, together with $(1_\huaV)_2=0$.

 Therefore we obtain  a  category {\rm \bf $2$Term-L$_\infty$} with $2$-term $L_\infty$-triple algebras as objects,  $2$-term $L_\infty$-triple algebras homomorphism as morphisms.

\section{  Lie triple 2-algebras}\label{sec:pre}

In this section, we define Lie triple 2-algebras, which are the categorification of Lie triple algebras, and show that the category of Lie triple 2-algebras is equivalent to the category of 2-term $L_\infty$-triple algebras.

Vector spaces can be categorified to $2$-vector spaces.
A $2$-vector space is a category in the category of vector spaces.
For the reference, see \cite{BC}.

Thus a $2$-vector space $C$ is a category with a vector space of
objects $C_0$ and a vector space of morphisms $C_1$, such that all
the structure maps are linear. Let $s,t:C_1\longrightarrow C_0$ be
the source and target maps respectively. Let $\cdot_\ve$ be the
composition of morphisms. Let $1:C_0\longrightarrow C_1$ be the unit map, i.e. for all $x\in C_0$, $1_x\in C_1$ is the identity morphism from $x$ to $x$.

It is well known that the category of 2-vector spaces is
equivalent to the category of 2-term complexes of vector spaces.
Roughly speaking, given a 2-vector space $C$, we have a 2-term complex of vector spaces
\begin{equation}\label{eq:complex}
\Ker(s)\stackrel{t}{\longrightarrow}C_0.
\end{equation}
Conversely, any 2-term complex of vector spaces
$\huaV:V_1\stackrel{\dM}{\longrightarrow}V_0$ gives rise to a
2-vector space of which the set of objects is $V_0$, the set of
morphisms is $V_0\oplus V_1$, the source map $s$ is given by
$s(v+m)=v$, and the target map $t$ is given by $t(v+m)=v+\dM m$,
where $v\in V_0,~m\in V_1.$ We denote the 2-vector space associated
to the 2-term complex of vector spaces
$\huaV:V_1\stackrel{\dM}{\longrightarrow}V_0$ by $\V$:
\begin{equation}\label{eqn:V}
\V=\begin{array}{c}
\V_1:=V_0\oplus V_1\\
\vcenter{\rlap{s }}~\Big\downarrow\Big\downarrow\vcenter{\rlap{t }}\\
\V_0:=V_0.
 \end{array}\end{equation}

\begin{defi}
A   Lie triple 2-algebra consists of:
\begin{itemize}
\item[$\bullet$] a $2$-vector spaces $L$;
\item[$\bullet$] a bilinear functor and a trilinear functor $[\cdot,\cdot]: L\times L\longrightarrow L$ and $[\cdot,\cdot,\cdot]: L\times L\times L\longrightarrow L$;
\item[$\bullet$] four multilinear natural isomorphism,
$$J_{x_1, x_2,x_3}:[x_1, [x_2,x_3]]+[x_3,x_2, x_1]\longrightarrow [[x_1, x_2],x_3]+ [x_2, [x_1,x_3]]+[x_1,x_2, x_3]+[x_3,x_1, x_2],$$
$$\widehat{J}_{x_1, x_2,x_3,y_1}: [x_1, [x_2,x_3],y_1]\longrightarrow [[x_1, x_2],x_3,y_1]+[x_2, [x_1,x_3],y_1],$$
$$\widetilde{J}_{x_1, x_2,y_1,y_2}: [x_1, x_2, [y_1, y_2]]\longrightarrow  [[x_1, x_2, y_1], y_2] + [y_1, [x_1, x_2, y_2]],$$
$J_{X,y_1,y_2,y_3}:[X, [y_1, y_2, y_3]] \longrightarrow  [[X, y_1], y_2, y_3] + [y_1, [X, y_2], y_3] + [y_1, y_2, [X, y_3]],$
\end{itemize}
such that for all $x_i,y_j,z_k\in L_0$, the following  identities hold:
\begin{itemize}
\item[$\rm(J_1)$]
\begin{align}
&[x_1,x_2,J_{y_1,y_2,y_3}](\widetilde{J}_{x_1,x_2,[y_1,y_2],y_3}+\widetilde{J}_{x_1,x_2,y_2,[y_1,y_3]}
+{J}_{X,y_1,y_2,y_3]}+{J}_{X,y_3,y_1,y_2]})\notag\\
&([\widetilde{J}_{x_1,x_2,y_1,y_2},y_3]+1+1+[y_2,\widetilde{J}_{x_1,x_2,y_1,y_3}])\notag\\
=&(\widetilde{J}_{x_1,x_2,y_1,[y_2,y_3]})(1+[y_1,\widetilde{J}_{x_1,x_2,y_2,y_3}]+{J}_{X,y_3,y_2,y_1]})\notag\\
&(J_{[x_1,x_2,y_1],y_2,y_3}+J_{y_1,[x_1,x_2,y_2],y_3}+J_{y_1,y_2,[x_1,x_2,y_3]}),
\end{align}
\item[$\rm(J_2)$]
\begin{align}
&[x_1,x_2,\widehat{J}_{y_1,y_2,y_3,z_1}](J_{X,[y_1,y_2],y_3,z_1}+J_{X,y_2,[y_1,y_3],z_1})\notag\\
&([\widetilde{J}_{x_1,x_2,y_1,y_2},y_3,z_1]+1+1+1+[y_2,\widetilde{J}_{x_1,x_2,y_1,y_3},z_1]+1)\notag\\
=&J_{X,y_1,[y_2,y_3],z_1}(\widehat{J}_{[X,y_1],y_2,y_3,z_1}+[y_1,\widetilde{J}_{x_1,x_2,y_2,y_3},z_1]+\widehat{J}_{y_1,y_2,y_3,[X,z_1]})\notag\\
&(1+1+\widehat{J}_{y_1,[x_1,x_2,y_2],y_3,z_1}+\widehat{J}_{y_1,y_2,[x_1,x_2,y_3],z_1}+1+1),
\end{align}
\item[$\rm(J_3)$]
\begin{align}
&[X,\widetilde{J}_{Y,z_1,z_2}] \Big(\widetilde{J}_{X, [Y,z_1],z_2}+\widetilde{J}_{X,z_1,[Y,z_2]}\Big) \Big([J_{X,Y,z_1},z_2]+ 1+1+[z_1,J_{X,Y,z_2}]\Big)\notag\\
=& J_{X,Y,[z_1,z_2]} \Big(1+[Y,\widetilde{J}_{X,z_1,z_2}]\Big) \Big({\widetilde{J}_{X\circ Y,z_1,z_2}+\widetilde{J}_{Y,[X,z_1],z_2}+\widetilde{J}_{Y,z_1,[X,z_2]}}\Big),
\end{align}
\item[$\rm(J_4)$]
\begin{align}
&[X,J_{Y,z_1,z_2,z_3}](J_{X,[Y,z_1],z_2,z_3}+J_{X,z_1,[Y,z_2],z_3}+J_{X,z_1,z_2,[Y,z_3]})\notag\\
&([J_{X,Y,z_1},z_2,z_3]+1+[z_1,J_{X,Y,z_2},z_3]+1+([z_1,z_2,J_{X,Y,z_3}])\notag\\
=&J_{X,Y,[z_1,z_2,z_3]}(1+1+[Y,J_{X,z_1,z_2,z_3}])(J_{[X,y_1],y_2,z_1,z_2,z_3}+J_{y_1,[X,y_2],z_1,z_2,z_3}\notag\\
&J_{Y,[X,z_1],z_2,z_3}+J_{Y,z_1,[X,z_2],z_3}+J_{Y,z_1,z_2,[X,z_3]}),
\end{align}
\end{itemize}
or, in terms of a commutative diagram,
$$\def\objectstyle{\scriptstyle}
  \def\labelstyle{\scriptstyle}
\xymatrix{
&[x_1,x_2,[y_1,[y_2,y_3]]+[y_3,y_2,y_1]]
\ar[dr]^{\widetilde{J}_{x_1,x_2,y_1,[y_2,y_3]}}\ar[dl]_{[x_1,x_2,J_{y_1,y_2,y_3}]}&\\
{\begin{aligned}&\scriptstyle[x_1,x_2,[[y_1,y_2],y_3]]+[x_1,x_2,[y_2,[y_1,y_3]]]\\[-.5em]
&\scriptstyle+ [x_1,x_2,[y_1,y_2,y_3]]+[x_1,x_2,[y_3,y_1,y_2]]
\end{aligned}}
\ar[dd]_{\widetilde{J}_{x_1,x_2,[y_1,y_2],y_3}}
\ar[dd]^{+\widetilde{J}_{x_1,x_2,y_2,[y_1,y_3]}+{J}_{X,y_1,y_2,y_3}+{J}_{X,y_3,y_1,y_2}}
&& {\begin{aligned}&\scriptstyle
[[x_1,x_2,y_1],[y_2,y_3]]+[y_1,[x_1,x_2,[y_2,y_3]]\\[-.5em]
&\scriptstyle+[x_1,x_2,[y_3,y_2,y_1]]
\end{aligned}}
\ar[dd]_{1+[y_1,\widetilde{J}_{x_1,x_2,y_2,y_3}]}\ar[dd]^{+{J}_{X,y_3,y_2,y_1}}\\
&&\\
{\begin{aligned}&\scriptstyle[[x_1,x_2,[y_1,y_2]],y_3]+ [[y_1,y_2,[x_1,x_2,y_3]]\\[-.5em]
&\scriptstyle+[[x_1,x_2,y_2],[y_1,y_3]]+[y_2,[x_1,x_2,[y_1,y_3]]]+\cdots
\end{aligned}}
\ar[dr]_{[\widetilde{J}_{x_1,x_2,y_1,y_2},y_3]+1+1+[y_2,\widetilde{J}_{x_1,x_2,y_1,y_3}]}&&
{\begin{aligned}&\scriptstyle [[x_1,x_2,y_1],[y_2,y_3]]+[y_1,[x_1,x_2,y_2],y_3]]\\[-.5em]
&\scriptstyle+[y_1,[y_2[x_1,x_2,y_3]]]+\cdots
\end{aligned}}
\ar[dl]^{J_{[x_1,x_2,y_1],y_2,y_3}+J_{y_1,[x_1,x_2,y_2],y_3}+J_{y_1,y_2,[x_1,x_2,y_3]}}\\
&R_1&}
\\ \\
$$
where $R_1$ is given by
$$
\begin{aligned}
R_1=&[[[x_1,x_2,y_1],y_2],y_3]+[[y_1, [x_1,x_2,y_2]],y_3]\\
&+ [y_1,y_2,[x_1,x_2,y_3]]+[[x_1,x_2,y_2],[y_1,y_3]]\\
&+[y_2,[[x_1,x_2,y_1],y_3]]+[y_2,[y_1,[x_1,x_2,y_3]]]\\
&+[[X,y_1], y_2,y_3]+[y_1,[X,y_2], y_3]\\
&+[y_1,y_2,[X,y_3]]+[[X,y_3], y_1,y_2]\\
&+[y_3,[X,y_1], y_2]+[y_3,y_1,[X,y_2]],
\end{aligned}
$$

$$\def\objectstyle{\scriptstyle}
  \def\labelstyle{\scriptstyle}
\xymatrix{
&[x_1,x_2,[y_1,[y_2,y_3],z_1]]\ar[dr]^{J_{X,y_1,[y_2,y_3],z_1}}\ar[dl]_{[x_1,x_2,\widehat{J}_{y_1,y_2,y_3,z_1}]}&\\
{\begin{aligned}&\scriptstyle [X,[[y_1,y_2],y_3,z_1]+[X,[y_2,[y_1,y_3],z_1]]\\[-.5em]
&\end{aligned}}
\ar[dd]_{J_{X,[y_1,y_2],y_3,z_1}}
\ar[dd]^{+J_{X,y_2,[y_1,y_3],z_1}}
&&{\begin{aligned}&\scriptstyle[[X,y_1],[y_2,y_3],z_1]]+[y_1,[X,[y_2,y_3]],z_1]]\\[-.5em]
&\scriptstyle +[y_1,[y_2,y_3],[X,z_1]]\end{aligned}}
\ar[dd]_{\widehat{J}_{[X,y_1],y_2,y_3,z_1}+[y_1,\widetilde{J}_{x_1,x_2,y_2,y_3},z_1]}
\ar[dd]^{+\widehat{J}_{y_1,y_2,y_3,[X,z_1]}}\\
&&\\
{\begin{aligned}&\scriptstyle [[X,[y_1,y_2]],y_3,z_1]+ [[y_1,y_2],[X,y_3],z_1]\\[-.5em]
&\scriptstyle +[[y_1,y_2],y_3,[X,z_1]]+[[X,y_2],[y_1,y_3],z_1]\\[-.5em]
&\scriptstyle +[y_2,[X,[y_1,y_3]],z_1]]+[y_2,[y_1,y_3],[X,z_1]]\\[-.5em]
\end{aligned}}
\ar[dr]_{[\widetilde{J}_{x_1,x_2,y_1,y_2},y_3,z_1]+1+1+1+[y_2,\widetilde{J}_{x_1,x_2,y_1,y_3},z_1]+1}&&
{\begin{aligned}&\scriptstyle[[X,y_1],y_2]],y_3,z_1]+ [y_2,[[X,y_1],y_3],z_1]\\[-.5em]
&\scriptstyle+[y_1,[[X,y_2],y_3],z_1]]+[y_1,[y_2,[X,y_3]],z_1]\\[-.5em]
&\scriptstyle+[[y_1,y_2],y_3,[X,z_1]]+[y_2,[y_1,y_3],[X,z_1]]\\[-.5em]
\end{aligned}}
\ar[dl]^{1+1+\widehat{J}_{y_1,[x_1,x_2,y_2],y_3,z_1}+\widehat{J}_{y_1,y_2,[x_1,x_2,y_3],z_1}+1+1}\\
&R_2&}
\\ \\
$$
where $R_2$ is given by
$$
\begin{aligned}
R_2=&[[[x_1, x_2, y_1], y_2], y_3, z_1] +[[y_1,[x_1, x_2, y_2], y_3, z_1] \\
&+[[y_1, y_2],[x_1, x_2, y_3], z_1]+[[y_1, y_2], y_3,[x_1, x_2, z_1]] \\
&+[[x_1, x_2, y_2],[y_1, y_3], z_1]] +[y_2,[[x_1, x_2, y_1], y_3, z_1]]]\\
&+[y_2,[y_1,[x_1, x_2, y_3]], z_1]]+[y_2,[y_1, y_3],[x_1, x_2, z_1]],
\end{aligned}
$$

$$\def\objectstyle{\scriptstyle}
  \def\labelstyle{\scriptstyle}
\xymatrix{
&[X,[y_1,y_2,[z_1,z_2]]]\ar[dr]^{J_{X,y_1,y_2,[z_1,z_2]}}\ar[dl]_{[X,\widetilde{J}_{Y,z_1,z_2}]}&\\
[X,[[Y,z_1],z_2]]+[X,[z_1,[Y,z_2]]]
\ar[dd]^{\widetilde{J}_{X,[Y,z_1],z_2}+\widetilde{J}_{X,z_1,[Y,z_2]}}
&&{\begin{aligned}&\scriptstyle[[X,y_1],y_2,[z_1,z_2]]+[y_1,[X,y_2],[z_1,z_2]]\\[-.5em]
&\scriptstyle\quad+[Y,[X,[z_1,z_2]]]\end{aligned}}
\ar[dd]_{1+1+[Y,\widetilde{J}_{X,z_1,z_2}]}\\
&&\\
{\begin{aligned}&\scriptstyle\quad[[X,[Y,z_1]],z_2]+ [[Y,z_1],[X,z_2]]\\[-.5em]
&\scriptstyle\quad+[[X,z_1],[Y,z_2]]+[z_1,[X,[Y,z_2]]]\end{aligned}}
\ar[dr]_{[J_{X,Y,z_1},z_2]+1+1+[z_1,J_{X,Y,z_2}]}&&
{\begin{aligned}&\scriptstyle [[X,y_1],y_2,[z_1,z_2]]+[y_1,[X,y_2],[z_1,z_2]]\\[-.5em]
&\scriptstyle\quad+[Y,[[X,z_1],z_2]] +[Y,[z_1,[X,z_2]]] \end{aligned}}
\ar[dl]^{\widetilde{J}_{X\circ Y,z_1,z_2}+\widetilde{J}_{Y,[X,z_1],z_2}+\widetilde{J}_{Y,z_1,[X,z_2]}}\\
&R_3&}
\\ \\
$$
where $R_3$ is given by
$$
\begin{aligned}
R_3=&\left[\left[\left[x_{1}, x_{2}, y_{1}\right], y_{2}, z_{1}\right], z_{2}\right] +\left[\left[y_{1},\left[x_{1}, x_{2}, y_{2}\right], z_{1}\right], z_{2}\right] \\
&+\left[\left[y_{1}, y_{2},\left[x_{1}, x_{2}, z_{1}\right]\right], z_{2}\right]+\left[\left[y_{1}, y_{2}, z_{1}\right],\left[x_{1}, x_{2}, z_{2}\right]\right] \\
&+\left[\left[x_{1}, x_{2}, z_{1}\right],\left[y_{1}, y_{2}, z_{2}\right]\right] +\left[z_{1},\left[\left[x_{1}, x_{2}, y_{1}], y_{2}, z_{2}\right]\right]\right]\\
&+\left[z_{1},\left[y_{1},\left[x_{1}, x_{2}, y_{2}\right], z_{2}\right]\right]+\left[z_{1},\left[y_{1}, y_{2},\left[x_{1}, x_{2}, z_{2}\right]\right]\right],
\end{aligned}
$$

%%%%%%%%%%%%%%%%%%%%%%%%
$$\def\objectstyle{\scriptstyle}
  \def\labelstyle{\scriptstyle}
\xymatrix{
&\scriptstyle [X,[Y,[z_1,z_2,z_3]]]
\ar[dr]^{J_{X,y_1,y_2,[z_1,z_2,z_3]}}\ar[dl]_{[X,J_{Y,z_1,z_2,z_3}]}&\\
{\begin{aligned}&\scriptstyle [X,[[Y,z_1],z_2,z_3]]+[X,[z_1,[Y,z_2],z_3]]\\[-.5em]
&\scriptstyle +[X,[z_1,z_2,[Y,z_3]]]
\end{aligned}}
\ar[dd]^{J_{X,[Y,z_1],z_2,z_3}+J_{X,z_1,[Y,z_2],z_3}+J_{X,z_1,z_2,[Y,z_3]}}
&&{\begin{aligned}
&\scriptstyle  [[X,y_1],y_2,[z_1,z_2,z_3]]+[y_1,[X,y_2],[z_1,z_2,z_3]]\\[-.5em]
&\scriptstyle +[y_1,y_2,[X,[z_1,z_2,z_3]]]
\end{aligned}}
\ar[dd]_{1+1+[Y,J_{X,z_1,z_2,z_3}]}\\
&&\\
{\begin{aligned}
&\scriptstyle[[X,[Y,z_1]],z_2,z_3]+[[Y,z_1],[X,z_2],z_3]\\[-.5em]
&\scriptstyle+[[Y,z_1],z_2,[X,z_3]]+[[X,z_1],[Y,z_2],z_3]\\[-.5em]
&\scriptstyle+[z_1,[X,[Y,z_2]],z_3]+[z_1,[Y,z_2],[X,z_3]]\\[-.5em]
&\scriptstyle+[[X,z_1],z_2,[Y,z_3]]+[z_1,[X,z_2],[Y,z_3]]\\[-.5em]
&\scriptstyle+[z_1,z_2,[X,[Y,z_3]]],
\end{aligned}}
\ar[dr]_{\qquad \qquad  \qquad \qquad \qquad P}&&
{\begin{aligned}
&\scriptstyle [[X,y_1],y_2,[z_1,z_2,z_3]]+[y_1,[X,y_2],[z_1,z_2,z_3]]\\[-.5em]
&\scriptstyle +[Y,[[X,z_1],z_2,z_3]]+[Y,[z_1,[X,z_2],z_3]]\\[-.5em]
&\scriptstyle +[Y,[z_1,z_2,[X,z_3]]],
\end{aligned}}
\ar[dl]^{Q}\\
&R_4&}
\\ \\
$$
where $P,Q$ is given by
\begin{eqnarray*}
P&=&[J_{X,y_1,y_2,z_1},z_2,z_3]+1+1+1+[z_1,J_{X,y_1,y_2,z_2},z_3]\\
&&+1+1+1+[z_1,z_2,J_{X,y_1,y_2,z_3}],\\
Q&=&J_{[X,y_1],y_2,z_1,z_2,z_3}+J_{y_1,[X,y_2],z_1,z_2,z_3}+J_{Y,[X,z_1],z_2,z_3}\\
&&+J_{Y,z_1,[X,z_2],z_3}+J_{Y,z_1,z_2,[X,z_3]},
\end{eqnarray*}
 and $R_4$ is given by
\begin{eqnarray*}
R_4&=&[[[X,y_1],y_2,z_1],z_2,z_3]+[[y_1,[X,y_2],z_1],z_2,z_3]+[[Y,[X,z_1]],z_2,z_3]\\
&&+[[Y,z_1],[X,z_2],z_3]+[[Y,z_1],z_2,[X,z_3]]+[[X,z_1],[Y,z_2],z_3]\\
&&+[z_1,[[X,y_1],y_2,z_2],z_3]+[z_1,[y_1,[X,y_2],z_2],z_3]+[z_1,[Y,[X,z_2]],z_3]\\
&&+[z_1,[Y,z_2],[X,z_3]]+[[X,z_1],z_2,[Y,z_3]]+[z_1,[X,z_2],[Y,z_3]]\\
&&+[z_1,z_2,[[X,y_1],y_2,z_3]]+[z_1,z_2,[y_1,[X,y_2],z_3]]+[z_1,z_2,[Y,[X,z_3]]].
\end{eqnarray*}
\end{defi}

\begin{defi}\label{defi:3liehomo}
  Given Lie triple 2-algebras $L$ and $L'$, a homomorphism $F:L\longrightarrow L'$ consists of:
\begin{itemize}
\item[$\bullet$] a linear functor $F$ from the underlying $2$-vector space of $L$ to that of $L',$ and
     \item[$\bullet$] a bilinear natural transformation
 $$F_2(x,y):[F_0(x),F_0(y)]' \longrightarrow F_0([x,y]),$$
  such that
\begin{eqnarray*}
  &&(F_1J_{x,y,z})F_2([x,y],z)([F_2(x,y),F_0(z)])\\
  &=&(F_2(x,[y,z])+F_2([x,z],y))([F_0(x),F_2(y,z)]+[F_2(x,z),F_0(y)])J_{F_0(x),F_0(y),F_0(z)},
\end{eqnarray*}
%\begin{eqnarray*}
%  &&(F_1J_{x,y,z})(F_2(x,[y,z])-F_2(y,[x,z]))([F_0(x),F_2(y,z)]-[F_0(y),F_2(x,z)])\\
%  &=&F_2([x,y],z)([F_2(x,y),F_0(z)])(J_{F_0(x),F_0(y),F_0(z)}),
%\end{eqnarray*}

\item[$\bullet$] a trilinear natural transformation
 $$F_3(x,y,z):[F_0(x),F_0(y),F_0(z)]'\longrightarrow F_0[x,y,z]$$
 such that
 \begin{eqnarray*}
   &&[F_0(x_1),F_0(x_2),F_3(x_3,x_4,x_5)] F_3(x_1, x_2, [x_3,x_4,x_5]) F_1(J_{x_1,x_2,x_3,x_4,x_5})\\
  &=& J'_{F_0(x_1),F_0(x_2),F_0(x_3),F_0(x_4),F_0(x_5)}([F_3(x_1,x_2,x_3),F_0(x_4),F_0(x_5)]\\
  &&+[F_0(x_4),,F_3(x_1,x_2,x_4),F_0(x_5),]+[F_0(x_3),F_0(x_4),F_3(x_1,x_2,x_5)])\\
  &&(F_3([x_1,x_2,x_3],x_4,x_5)+F_3(x_3,[x_1,x_2,x_4],x_5)+F_3(x_3,x_4,[x_1,x_2,x_5])).
   \end{eqnarray*}
\end{itemize}
\end{defi}
The identity homomorphism $id_{L}:L\longrightarrow L$ has the identity functor as its underlying functor, together with an identity natural transformation as $(id_L)_2.$ Let $L,L'$ and $L''$ be Lie triple 2-algebras, the composite of Lie triple 2-algebra homomorphisms $F:L\longrightarrow L'$ and $G:L'\longrightarrow L''$ which we denote by $G\circ F$, is given by letting the functor $((G\circ F)_0,(G\circ F)_1)$ be the usual composition of $(G_0,G_1)$ and $(F_0,F_1)$, and letting $(G\circ F)_2$, $(G\circ F)_3$ be the following composite: $(G\circ F)_2(x,y)=G_1(F_2(x,y))\circ G_2(F_0(x),F_0(y))$,
$(G\circ F)_3(x,y,z)=G_1(F_3(x,y,z))\circ G_3(F_0(x),F_0(y),F_0(z))$.

Thus we obtain a category   {\rm \bf Lie$2$Triple} with   Lie triple 2-algebras as objects, Lie triple 2-algebra homomorphisms  as morphisms.

Now we  establish the equivalence between the category of Lie triple 2-algebras and that of $2$-term $L_\infty$-triple algebras.

\begin{thm}\label{mainthm2}
The categories {\rm \bf $2$Term-L$_\infty$} and {\rm \bf Lie$2$Triple} are equivalent.
\end{thm}

\begin{proof} First we construct a 2-functor $T:$ {\rm \bf $2$Term-L$_\infty$} $\longrightarrow$ {\rm \bf Lie$2$Triple}.
Given a $2$-term $L_\infty$-triple algebra $\huaV=(V_1,V_0,\dM,l_2,l_3,l_4,l_5)$, we have a $2$-vector space $L$ via \eqref{eqn:V}.
Let $L_0=V_0, L_1=V_0\oplus V_1$ and the source and the target map be given by $s(x+f)=x$ and $t(x+f)=x+\dM f$. We define
  a skew-symmetric bilinear functor $[\cdot,\cdot]:L\times L\longrightarrow L$ and trilinear functor $[\cdot,\cdot,\cdot]:L\times L\times L\longrightarrow L$ by
  \begin{eqnarray*}
    [x+f,y+g]&=&[x,y]+[x,g]+[f,y]+[f,\dM g],\\
  {[x+f,y+g,z+h]}&=&l_3(x,y,z)+l_3(x,y,h)+l_3(x,g,z)+l_3(f,y,z)\\
  &&+l_3(\dM f,g,z)+l_3(\dM f,y,h)+l_3(x,\dM g,h)+l_3(\dM f,\dM g,h),
  \end{eqnarray*}
and the natural isomorphism by
\begin{eqnarray*}
J_{x_1,x_2,x_3}&=&([x_1,[x_2,x_3]],l_3(x_1,x_2,x_3)),\\
\widehat{J}_{x_1,x_2,x_3,y_1}&=&([x_1,x_2,[x_3,y_1]],\widehat{l}_4(x_1,x_2,x_3,y_1)),\\
\widetilde{J}_{x_1,x_2,y_1,y_2}&=&([x_1,x_2,[y_1,y_2]],\widetilde{l}_4(x_1,x_2,y_1,y_2)),\\
J_{x_1,x_2,y_1,y_2,y_3}&=&([x_1,x_2,[y_1,y_2,y_3]],l_5(x_1,x_2,y_1,y_2,y_3)).
\end{eqnarray*}

The source of $J_{x_1,x_2,x_3},  \widehat{J}_{x_1,x_2,x_3,y_1},  \widetilde{J}_{x_1,x_2,y_1,y_2}, J_{x_1,x_2,y_1,y_2,y_3}$ are
$$[x_1,[x_2,x_3]],\quad [x_1,x_2,[x_3,y_1]],\quad  [x_1,x_2,[y_1,y_2]],\quad [x_1,x_2,[y_1,y_2,y_3]].$$
The target of them are
\begin{eqnarray*}
J_{x_1,x_2,x_3}&=&[x_1,[x_2,x_3]]+l_3(x_1,x_2,x_3)\\
&=&[[x_1, x_2],x_3]+ [x_2, [x_1,x_3]],\\
\widehat{J}_{x_1,x_2,x_3,y_1}&=&[x_1,x_2,[x_3,y_1]]+\widehat{l}_4(x_1,x_2,x_3,y_1)\\
&=& [[x_1, x_2],x_3],y_1]+[[x_2, [x_1,x_3],y_1],\\
\widetilde{J}_{x_1,x_2,y_1,y_2}&=&[x_1,x_2,[y_1,y_2]]+\widetilde{l}_4(x_1,x_2,y_1,y_2)\\
&=&[[ x_1, x_2, y_1], y_2] + [y_1, [ x_1, x_2, y_2]],\\
J_{x_1,x_2,y_1,y_2,y_3}&=&[x_1,x_2,[y_1,y_2,y_3]]+l_5(x_1,x_2,y_1,y_2,y_3)\\
&=& [[x_1,x_2, y_1], y_2, y_3] + [y_1, [x_1,x_2, y_2], y_3] + [y_1, y_2, [x_1,x_2, y_3]].
\end{eqnarray*}
 By Conditions in Definition \ref{lem:2term 3L}, we deduce that the natural transformation identity holds.
This completes the construction of a Lie triple 2-algebra $L=\Phi(\huaV)$ from a $2$-term $L_\infty$-triple algebra $\huaV.$

Next we construct a Lie triple 2-algebra homomorphism $\Phi(\phi):\Phi (\huaV)\longrightarrow \Phi(\huaV')$ from a $L_\infty$-triple algebra homomorphism $\phi=(\phi_0,\phi_1,\phi_2,\phi_3):\huaV\longrightarrow \huaV'$ between $2$-term $L_\infty$-triple algebras.
Let $T(\huaV)=L$ and $T(\huaV')=L'.$ We define the underlying linear functor of $T(\phi)=F$ with $F_0=\phi_0$, $F_1=\phi_0\oplus\phi_1.$ Define $F_2:V_0\times V_0\longrightarrow V_0'\oplus V_1'$ and $F_3:V_0\times V_0\times V_0\longrightarrow V_0'\oplus V_1'$ by
\begin{eqnarray*}
F_2(x_1,x_2,)&=&([\phi_0(x_1),\phi_0(x_2)]',\phi_2(x_1,x_2)),\\
F_3(x_1,x_2,x_3)&=&([\phi_0(x_1),\phi_0(x_2),\phi_0(x_3)]',\phi_3(x_1,x_2,x_3)).
\end{eqnarray*}
Then $F_2(x_1,x_2)$ is a natural isomorphism from $[F_0(x_1),F_0(x_2)]'$ to $F_0[x_1,x_2]$. $F_3(x_1,x_2,x_3)$ is a natural isomorphism from $[F_0(x_1),F_0(x_2),F_0(x_3)]'$ to $F_0[x_1,x_2,x_3]$, and $F=(F_0,F_1,F_2,F_3)$ is a homomorphism from $L$ to $L'$.
We can also prove that  $\Phi$ preserve identities and composition of homomorphisms. Thus $\Phi$ is a functor.

Conversely, we construct a functor $\Psi:$ {\rm \bf Lie$2$Triple} $\longrightarrow$ {\rm \bf $2$Term-L$_\infty.$}
Given a Lie triple 2-algebra $L$, we obtain a complex of vector spaces  $\huaV=\Psi(L)$ via \eqref{eq:complex}. More precisely, $V_1=\ker(s),V_0=L_0$ and $\dM=t|_{\ker(s)}$.
For all $x_i\in V_0=L_0$ and $f,g,h\in V_1\subseteq L_1$,
we define  $l_3$ and $l_5$ as follows:
\begin{itemize}
\item[$\rm(1)$]$l_3(x_1,x_2,x_3)=[1_{x_1},1_{x_2},1_{x_3}],$
\item[$\rm(2)$]$l_3(x_1,x_2,h)=[1_{x_1},1_{x_2},h],$
\item[$\rm(3)$]$l_3(x_1,f,h)=0,~l_3(f,h,g)=0,$
\item[$\rm(4)$]$\widehat{l}_4(x_1,x_2,x_3,y_1)=\widehat{J}_{x_1,x_2,x_3,y_1}-1_{s(\widehat{J}_{x_1,x_2,x_3,y_1})},$
\item[$\rm(4)$] $\widetilde{l}_4(x_1,x_2,y_1,y_2)=\widetilde{J}_{x_1,x_2,y_1,y_2}-1_{s(\widetilde{J}_{x_1,x_2,y_1,y_2})},$
\item[$\rm(5)$]$l_5(x_1,x_2,x_3,x_4,x_5)=J_{x_1,x_2,x_3,x_4,x_5}-1_{s(J_{x_1,x_2,x_3,x_4,x_5})}.$
\end{itemize}
The
various conditions of $L$ being a Lie triple 2-algebra imply that $\huaV=(V_1,V_0,\dM,l_3,\widehat{l}_4, \widetilde{l}_4,l_5)$ is 2-term $L_\infty$-triple algebra.
This completes the construction of a $2$-term $L_\infty$-triple algebra $\huaV=\Psi(L)$ from a Lie triple 2-algebra $L$.

Let $L$ and $L'$ be Lie triple 2-algebras, and $F=(F_0,F_1,F_2,F_3):L\longrightarrow L'$ a homomorphism. Let $\Psi(L)=\huaV$ and $\Psi(L')=\huaV'$. We are going to construct a homomorphism $\phi=\Psi(F):\huaV\longrightarrow \huaV'$.
Let $\phi_0=F_0$, $\phi_1=F_1|_{\ker(s)}$. Define $\phi_2:V_0\times V_0\times V_0\longrightarrow V'_1$ by
 \begin{eqnarray*}
 \phi_2(x_1,x_2)&=&F_3(x_1,x_2)-1_{s(F_3(x_1,x_2))},\\
\phi_3(x_1,x_2,x_3)&=&F_3(x_1,x_2,x_3)-1_{s(F_3(x_1,x_2,x_3))}.
\end{eqnarray*}
Then we get
 \begin{eqnarray*}
 \dM' \phi_2(x_1,x_2)&=&\phi_0(l_2(x_1,x_2))-l_2'(\phi_0(x_1),\phi_0(x_2)),\\
\dM' \phi_3(x_1,x_2,x_3)%&=&(t'-s')F_3(x_1,x_2,x_3)\\
&=&\phi_0(l_3(x_1,x_2,x_3))-l_3'(\phi_0(x_1),\phi_0(x_2),\phi_0(x_3)).
\end{eqnarray*}
Thus, $\phi=\Psi(F)$ is a homomorphism between 2-term $L_\infty$-triple algebras.

Finally, one see that  $\Psi$ and $\Psi$ are inverse to each other. The proof is completed.
\end{proof}

By Theorem \ref{mainthm2}, we see that the concept of Lie triple 2-algebras and 2-term $L_\infty$-triple algebras are equivalent. Thus, we will call a 2-term $L_\infty$-triple algebra as a Lie triple 2-algebra in the sequel.

\section{Skeletal and strict Lie triple 2-algebras}
In this section, we study two special cases of Lie triple 2-algebras $\huaV=(V_1,V_0,\dM,l_3,\widetilde{l}_4, \widehat{l}_4,l_5)$.
The first case is the {\em skeletal} Lie triple 2-algebras when the  2-term $L_\infty$-triple algebras   if $\dM=0$.
The second one is the crossed module of Lie triple algebras which is equivalent to strict case of a Lie triple 2-algebra when $(l_3,\widetilde{l}_4, \widehat{l}_4,l_5)=0$.
We will construct crossed module of Lie triple algebras by using crossed module of Leibniz algebras and reductive crossed module of Lie algebras.

\subsection{Skeletal  Lie triple 2-algebras}

\begin{thm}
  There is a one-to-one correspondence between skeletal 2-term Lie triple 2-algebras and druples $(T,V,\rho,\theta,\DD)$,
  where $T$ is a Lie triple algebras, $V$ is a vector space,   $(\rho,\theta,\DD)$ is a  representation of $T$ on $V$,
  $(l_3,\widetilde{l}_4, \widehat{l}_4,l_5)$  is a (3,4,4,5)-cocycle of the Lie triple algebra $T$ with  coefficients in the representation $V$.
\end{thm}
\begin{proof}
%For any skeletal 2-term $L_\infty$-triple algebra $\huaV$,  we have seen that $V_1$ is a representation of $V_0$.
 Let $\huaV$ be a skeletal  2-term $L_\infty$-triple algebra.
By Conditions $(e_1)$--$(e_4)$, we see that $V_0$ is  a Lie triple algebra.
Define the representation maps $D(x_1,x_2)$, $\theta(x_1,x_2)$, $\rho(x_1)$  by
\begin{eqnarray*}
D(x_1,x_2)(u):=[x_1,x_2,u],\quad \theta(x_1,x_2)(u):=[u,x_1,x_2],\quad \rho(x_1)(u):=[x_1,u].
\end{eqnarray*}
for $x_1,x_2\in V_0$ and $u\in V_1$. Then it is easy to verify that $V_1$ is a representation of $V_0$.
Now by Conditions $(f_1)$--$(f_4)$, we have   $(l_3,\widetilde{l}_4, \widehat{l}_4,l_5)$  is a (3,4,4,5)-cocycle of the Lie triple algebra $V_0$ with  coefficients in the representation $V_1$.

Conversely, given a Lie triple algebra $T$, a representation  $(\rho,\DD,\theta)$ of $T$ on a vector space $V$, and a (3,4,4,5)-cocycle
$(l_3,\widetilde{l}_4, \widehat{l}_4,l_5)$ on $T$ with values in $V$, we define $V_0 = T, V_1 = V , \dM = 0$, and define maps
\begin{eqnarray*}
[x_1,x_2,u]:=D(x_1,x_2)(u),\quad [u,x_1,x_2]:=\theta(x_1,x_2)(u),\quad [x_1,u]:=\rho(x_1)(u).
\end{eqnarray*}
Then, it is straightforward to see that $(V_1, V_0, l_3,\widetilde{l}_4, \widehat{l}_4,l_5)$ is a skeletal Lie triple 2-algebra.
This finishes the  proof.
\end{proof}

%\begin{defi}
%A quadratic Lie-Yamaguti algebra is a Lie-Yamaguti algebra $T$ equipped with a nondegenerate symmetric bilinear form $\mathcal{B}\in T\otimes T$ satisfying the following invariant conditions
%\begin{eqnarray*}
%\mathcal{B}([x_1, x_2], x_3) =-\mathcal{B}(x_2, [x_1, x_3]),\\
%\mathcal{B}([x_1, x_2, x_3], x_4) =-\mathcal{B}(x_3, [x_4,x_3, x_2]).
%\end{eqnarray*}
%We denote a quadratic Lie-Yamaguti algebra by $(T,\mathcal{B})$.
%\end{defi}
%\begin{exa}
%Given a quadratic Lie-Yamaguti algebra $(T,\mathcal{B})$, we construct a Lie-Yamaguti 2-algebra as follows.
%Let $\mathcal{V}_1=\mathbb{C}, \mathcal{V}_0=T, \dM=0$, and define $l_3,l_4$ by
%\begin{equation}\label{eqn:l2l3string}
%l_3(x_1, x_2, x_3)=\mathcal{B}([x_1, x_2], x_3),\quad l_4(x_1, x_2, x_3, x_4)=\mathcal{B}([x_1, x_2, x_3], x_4).
%\end{equation}
%We call it the \emph{string Lie-triple 2-algebra}.
%\end{exa}

\subsection{Crossed module of Lie triple algebras}

First, we introduce the notion of crossed modules of Lie triple algebras and establish their equivalence to strict Lie triple 2-algebras.

\begin{defi}Let $\left({T},[\cdot, \cdot]_{{T}},[\cdot, \cdot]_{{T}}\right)$ and $\left({V},[\cdot, \cdot]_{{V}},[\cdot, \cdot, \cdot]_{{V}}\right)$ be two Lie-Yamaguti algebras. We define $(\rho, \theta, D)$ to be an action of ${T}$ on ${V}$ if there is a Lie-Yamaguti algebra structure on the direct sum ${T} \oplus {V}$ defined by
$$
\begin{aligned}
{[x+u, y+v]}& =[x, y]_{{T}}+\rho(x) (v)-\rho(y) (u)+[u, v]_{{V}}, \\
{[x+u, y+v, z+w ]} & =[x, y, z]_{{T}}+D(x, y) (w)+\theta(y, z) (u)-\theta(x, z) (v)+[u, v, w]_{{V}},
\end{aligned}
$$
for all $x, y, z \in {T}$ and $u, v, w \in {V}$.
This is called the semidirect product Lie-Yamaguti algebra with respect to the action $(\rho, \theta, D)$, and is denoted by ${T} \ltimes_{\rho, \theta, D} {V}$.
\end{defi}

\begin{defi}
A crossed module of Lie triple algebras is a pair of  Lie triple algebras $T$ and $V$  together with an  action $(\rho,\DD,\theta)$ of ${T}$ on ${V}$ and a homomorphism between them $\partial:{V}\longrightarrow{T}$, such that  the following equalities hold:  $\forall  x,y,z\in {T},u,v,w\in{V}$,
\begin{eqnarray}
  \label{eq:cmc01}\partial(\rho(x)(u))&=&[x,\partial(u)]_{{T}},\\
    \label{eq:cmc1}\partial(\DD(x,y)(u))&=&[x,y,\partial(u)]_{{T}},\quad \partial(\theta(x,y)(u))=[\partial(u),x,y]_{{T}},\\
  \label{eq:cmc02}[u,v]_{{V}}&=&\rho(\partial(u))(v),\\%=-\rho(\partial(v))(u),\\
  \label{eq:cmc2}[u,v,w]_{{V}}&=&\DD(\partial(u),\partial(v))(w)=\theta(\partial(v),\partial(w))(u),\\%=-\theta(\partial(u),\partial(w))(v),\\
  \label{eq:cmc3}\DD(x,\partial(u))(v)&=&-\theta(x,\partial(v))(u).%\quad \DD(\partial(u),x)(v)=-\theta(\partial(v),x)(u).
\end{eqnarray}
\end{defi}

\begin{thm}
There is a one-to-one correspondence between strict Lie triple 2-algebras and crossed modules of Lie triple algebras.
\end{thm}
\pf  Let $(V_1,V_0,\dM)$ be a strict Lie triple 2-algebras. Define ${V}=V_1,{T}=V_0$.
Then it is easy to see from definition that $T$ is a Lie triple algebra.
We define the following two operations on ${V}$ by
\begin{eqnarray}
~[u,v]_{{V}}&=&[\dM u, v]=[u,\dM v],\\
~[u,v,w]_{{V}}&=&[\dM u,\dM v,w]=[\dM u,v,\dM w]=[u,\dM v,\dM w],
\end{eqnarray}
It can be verified that $V$ is a Lie triple algebra.
Let $\partial=\dM,$ by Condition (a)--(d) in Definition \ref{lem:2term 3L}, we have
$$\partial[u,v]_{{V}}=\dM [\dM u, v]= [\dM u, \dM v]=[\partial(u),\partial(v)]_{{T}},$$
and
$$\partial[u,v,w]_{{V}}=\dM [\dM u,\dM v,w]=[\dM u,\dM v,\dM w]=[\partial(u),\partial(v),\partial(w)]_{{T}},$$
which implies that $\partial$ is a homomorphism of Lie triple algebras.
Define maps $\rho:{T} \to\End({V})$  ${\DD}:{T}\times T \to\End({V})$ and $\theta:{T}\times T \to \End({V})$ by
$$\rho(x)(v)=[x,v], \quad {\DD}(x,y)(v)=[x,y,v],\quad \theta(x,y)(v)=[v,x,y].$$
By Conditions $(e_1)$--$(e_4)$  in Definition \ref{lem:2term 3L}, we can obtain that $(V,\rho,\DD,\theta)$ is a representation of $T$. Furthermore, we have
\begin{eqnarray*}
\partial(\rho(x)(u))&=&\dM[x,u]=[x,\dM(u)]=[x,\partial(u)]_{{T}},\\
{[u,v]_{V}}&=&[\dM u, v]=[u,\dM v]=\rho(\partial(u)(v)=-\rho(\partial(v)(u)\\
  \partial({\DD}(x,y)(v))&=&\dM ({\DD}(x,y)(v))=\dM [x,y,v]=[x,y,\dM v]= [x,y,\partial(v)]_{{T}},\\
 {[u,v,w]_{{V}}}&=&[\dM u,\dM v,w]=  {\DD}(\partial(u),\partial(v))(w),\\
  {\DD}(x,\partial(u))(v)&=&[x,\dM(u),v]=[x,u,\dM(v)]=-[u,x,\partial(v)]=-\theta(x,\partial(v))(u).
\end{eqnarray*}
Other conditions in \eqref{eq:cmc01}--\eqref{eq:cmc3} can be verified similarly.
Therefore we obtain a crossed module of Lie triple 2-algebras.

Conversely, a crossed module of Lie triple 2-algebras gives rise to a $2$-term $L_\infty$-triple algebra, in which $V_1={V},$  $V_0={T}$, $\dM=\partial$.
The crossed module conditions give various conditions for $2$-term $L_\infty$-triple algebras. We omit the details.
\qed\vspace{3mm}

There are two constructions of Lie triple algebras given in  \cite{KW}. The first one is  the existence of Lie triple algebra on any Leibniz algebra.
Let $({L},\cdot)$ be a Leibniz algebra with a binary operation $\cdot:{L}\times {L}\to {L}$ satisfying
the following Leibniz identity:
$$x\cdot (y\cdot z)=(x\cdot y)\cdot z+y\cdot (x\cdot z).$$
Define bilinear maps $[\cdot,\cdot] : {L} \times {L} \rightarrow {L}$ and trilinear maps $[\cdot,\cdot,\cdot] : {L} \times {L} \times {L}\rightarrow {L}$  by
$$
\begin{aligned}
{[x, y]}& =x\cdot y-y\cdot x,\\
[x, y, z]&=as(y,x,z)-as(x,y,z)=-(x\cdot y)\cdot z
\end{aligned}
$$
where $as(x,y,z)=(x\cdot y)\cdot z-x\cdot (y\cdot z)$ is the associator for $x, y, z \in {L}$. It is proved in \cite{KW} that $({L},[\cdot, \cdot],[\cdot, \cdot, \cdot])$ is a Lie triple algebra.

\begin{defi} \label{def:cmLeib}
A \emph{crossed module} of  Leibniz algebras is a triple of the form $({V}, {L},\varphi)$,
where ${V}$ and ${L}$ are  Leibniz algebras together with an  action of ${L}$ on ${V}$ and  a Leibniz  algebra homomorphism $\varphi \colon {V} \to {L}$ such that the following identities hold:
\begin{itemize}
\item[$\bullet$]{\rm(L1)}\quad $\varphi(x\triangleright u)=x\cdot\varphi(u),\quad \varphi(u\trl x)=\varphi(u)\cdot x,$
\item[$\bullet$]{\rm(L2)}\quad $\varphi(u)\trr v=u\cdot v=u\trl \varphi(v)$,
\end{itemize}
for all $u, v\in {V}, x \in {L}$.
\end{defi}

Now we construct crossed module of Lie triple algebras from crossed module of Leibniz algebras.

\begin{thm}\label{thm:c2}
Let $({V}, {L},\varphi)$ a crossed module of  Leibniz algebras. Then we obtain a crossed module of   Lie triple algebras $\partial=\varphi:{V}\to {L}$ by  representation of ${L}$ on ${V}$ through the following maps
$$
\begin{aligned}
{\rho(x)(u)}& =x\trr u-u\trl x,\\
D(x, y)(u)&=as(y,x,u)-as(x,y,u)=-(x\cdot y)\trr u,\\
\theta(x, y)(u)&=as(x,u,y)-as(u,x,y)=-(u\trl x)\trl y,
\end{aligned}
$$
where $as(x,y,u)=(x\cdot y)\trr u-x\trr (y\trr u)$ is the associator for $x, y\in {L}, u\in {V}$.
\end{thm}

\proof We verify the conditions in \eqref{eq:cmc01}--\eqref{eq:cmc3} as follows.
Firstly, we have
\begin{eqnarray*}
\partial(\rho(x)(u))&=&\vphi(x\trr u-u\trl x)=x\cdot \vphi(u)-\vphi(u)\cdot x= [x,\partial(u)]_{{{L}}},
\end{eqnarray*}
\begin{eqnarray*}
\partial(\DD(x,y)(u))
&=&\vphi((y\cdot x)\trr v-y\trr (x\trr u)-(x\cdot y)\trr u+x\trr (y\trr u))\\
&=&(y\cdot x)\trr \vphi(u)-y\trr (x\trr  \vphi(u))-(x\cdot y)\trr  \vphi(u)+x\trr (y\trr  \vphi(u)))\\
&=&[x,y,\vphi(u)]_{{{L}}}=[x,y,\partial(u)]_{{{L}}},
\end{eqnarray*}
\begin{eqnarray*}
 \partial(\theta(x,y)(u))&=&-\vphi((u\trl x)\trl y)=-\vphi(u\trl x)\cdot y\\
  &=&-(\vphi (u)\cdot x)\cdot y=[\partial(u),x,y]_{{{L}}}.
\end{eqnarray*}
Next, we compute
\begin{eqnarray*}
{[u,v]}_{{V}}&=&u\cdot v-v\cdot u=\partial(u)\trr v-v\trl\vphi(u)=\rho(\partial(u))(v),
\end{eqnarray*}
\begin{eqnarray*}
{[u,v,w]}_{{V}}&=&-(u\cdot v)\cdot w=-\varphi(\varphi(u)\trr v)\trr w\\
&=& -(\varphi(u)\cdot \varphi(v))\trr w=-(\partial(u)\cdot \partial(v))\trr w\\
&=&\DD(\partial(u),\partial(v))(w),
\end{eqnarray*}
Similarly, we get
\begin{eqnarray*}
{[u,v,w]}_{{V}}&=&-(u\cdot v)\cdot w= -(u\trl\varphi(v))\trl \varphi(w)\\
&=&-(u\trl \partial(v))\trl \partial(w)=\theta(\partial(v),\partial(w))(u).
\end{eqnarray*}
Finally, we have
\begin{eqnarray*}
\DD(x,\partial(u))(v)&=&-(x\cdot  \partial(u))\trr v=-\theta(x,\partial(v))(u).
\end{eqnarray*}
Thus we obtain a crossed module of Lie triple algebras from crossed module of Leibniz algebras.
\qed

The second construction of Lie triple algebras is the one from  reductive Lie algebras.
Let $\mathfrak{g}$ be a Lie algebra with a reductive decomposition $\mathfrak{g}=\mathfrak{h} \oplus \mathfrak{m}$, i.e. $[\mathfrak{h}, \mathfrak{h}] \subseteq \mathfrak{h}$ and $[\mathfrak{h}, \mathfrak{m}] \subseteq \mathfrak{m}$. On $\mathfrak{m}$, define bilinear map $[\cdot,\cdot]_\mathfrak{m}: \mathfrak{m} \times \mathfrak{m} \rightarrow \mathfrak{m}$ and trilinear map  $[\cdot,\cdot,\cdot]_\mathfrak{m}: \mathfrak{m} \times \mathfrak{m} \times \mathfrak{m} \rightarrow   \mathfrak{m}$ by the projections of the Lie bracket:
$$
\begin{aligned}
{[x, y]}_\mathfrak{m} :=\pi_{\mathfrak{m}}([x, y]_{\mathfrak{g}}), \quad{[x, y, z]}_\mathfrak{m}:=[\pi_{\mathfrak{h}}([x, y]_{\mathfrak{g}}), z]
\end{aligned}
$$
where $\pi_{\mathfrak{m}}:\g\to \m, \pi_{\mathfrak{h}}:\g\to \h$ are the projection map for $x, y, z \in \mathfrak{m}$. It is proved in \cite{KW} that $(\mathfrak{m},{[\cdot, \cdot]}_\m,{[\cdot, \cdot, \cdot]}_\m)$ is a Lie triple algebra.

\begin{defi} \label{def:cmlie}
A \emph{crossed module} of  Lie algebras is a triple $({V}, \mathfrak{g},\varphi)$,
where ${V}$ and $\mathfrak{g}$ are Lie algebras together with an  action of $\mathfrak{g}$ on ${V}$,   $\varphi \colon {V} \to \mathfrak{g}$ is a Lie  algebra homomorphism such that the following identities hold:
\begin{itemize}
\item[$\bullet$]{\rm(C1)}\quad $\varphi(x\triangleright u)=[x,\varphi(u)]_{\g},$
\item[$\bullet$]{\rm(C2)}\quad $\varphi(u)\trr v=[u, v]_{V}$.
\end{itemize}
for all $u, v\in {V}, x \in \mathfrak{g}$.
\end{defi}

\begin{defi} \label{def:cmlie}
A \emph{reductive crossed module} of  Lie algebras is a crossed module of Lie algebras $({V}, \mathfrak{g},\varphi)$,
where ${V}=V_1\oplus V_2$ and $\mathfrak{g}=\mathfrak{h}\oplus\mathfrak{m}$ are reductive Lie algebras such that $\varphi$ is decomposed into
$\varphi=\varphi_1\oplus \varphi_2$: $V\xrightarrow{\vphi} \g=(V_1\xrightarrow{\vphi_1}  \h)\oplus (V_1\xrightarrow{\vphi_2} \m)$
and the following identities hold:
\begin{itemize}
\item[$\bullet$]{\rm(R1)}\quad $({V_1}, \mathfrak{h},\varphi_1)$ is a subcrossed module,
\item[$\bullet$]{\rm(R2)}\quad $ \mathfrak{h}\trr V_2\subset V_2$,
%\item[$\bullet$]{\rm(R3)}\quad $\vphi(x\trr u)=[x, \vphi_2(u)]_\g$,
%\item[$\bullet$]{\rm(R4)}\quad $[u,v]_V=\varphi_2(u)\trr v$
\end{itemize}
for all $x, y\in \mathfrak{m}, u, v\in {V_2}$.
\end{defi}

\begin{thm}\label{thm:c2}
Let $({V}, \mathfrak{g},\varphi)$ a reductive crossed module of  Lie algebras. Then we get a crossed module of   Lie triple algebras $\partial=\varphi_2:{V_2}\to \mathfrak{m}$ by  representation of $\mathfrak{m}$ on ${V_2}$ through the following maps
$$
\begin{aligned}
{\rho(x)(u)}& =\pi_{V_2}(x\trr u),\quad D(x, y)(u)&=\pi_{\h}([x,y]_{\g})\trr u,\\
\theta(x, y)(u)&=\pi_{{V_2}}(y\trr\pi_{{V_2}}(x\trr u)),
\end{aligned}
$$
for all $x, y\in \mathfrak{m}, u, v\in {V_2}$.
\end{thm}

\proof We verify the conditions in \eqref{eq:cmc01}--\eqref{eq:cmc3} as follows.
For \eqref{eq:cmc01}, we have
\begin{eqnarray*}
\partial(\rho(x)(u))&=&\vphi_2(\pi_{V_2}(x\trr u))=\pi_{\m}\vphi(x\trr u)\\
&=&\pi_{\m}([x, \vphi_2(u)]_\g)=[x,\vphi_2(u)]_{{\m}}\\
&=& [x,\partial(u)]_{{\m}}.
\end{eqnarray*}
A direct computation shows that
\begin{eqnarray*}
\partial(\DD(x,y)(u))&=&\vphi_2(\pi_{\h}([x,y])\trr u)=[\pi_{\h}([x,y]),\vphi_2(u)]_{V_2}\\
&=&[x,y,\partial(u)]_{V_2},
\end{eqnarray*}
\begin{eqnarray*}
\partial(\theta(x,y)(u))&=&\vphi_2(\pi_{{V_2}}(y\trr\pi_{{V_2}}(x\trr u)))=[y,\vphi_2(\pi_{{V_2}}(x\trr u))]\\
&=&[y,[x, \vphi_2(u))]=[\pi_{\h}([\vphi_2(u),x]), y]_{V_2}\\
&=&[\partial(u),x,y]_{V_2}.
\end{eqnarray*}
Finally, we get
\begin{eqnarray*}
{[u,v]}_{{V_2}}&=&\pi_{V_2}([u,v]_V)=\pi_{V_2}(\varphi_2(u)\trr v)=\rho(\partial(u))(v),
\end{eqnarray*}
\begin{eqnarray*}
{[u,v,w]}_{V_2}&=&[\pi_{V_1}([u,v]_V), w]=\pi_{\h}([\vphi_2(u),\vphi_2(v)])\trr w\\
&=&\DD(\partial(u),\partial(v))(w)=\pi_{{V_2}}(\vphi_2(v)\trr\pi_{{V_2}}(\vphi_2(w)\trr u))\\
&=&\theta(\partial(v),\partial(w))(u).
\end{eqnarray*}
Thus we obtain a crossed module of Lie triple algebras from {reductive crossed module} of  Lie algebras.
\qed

%\begin{exa}
%For the hemisemidirect product $\mathcal{E}=\mathfrak{h} \ltimes_H V$, a short calculation using (5.3) and (2.5) shows that the trilinear product in the associated Lie-Yamaguti algebra $(\mathcal{E}, \llbracket \cdot, \rrbracket,\{\cdot, \cdot, \cdot\})$ is given by
%$$
%\{(\xi, x),(\eta, y),(\zeta, z)\}=-([[\xi, \eta], \zeta],[\xi, \eta] z)
%$$
%for $\xi, \eta, \zeta \in \mathfrak{h}, x, y, z \in V$.
%\end{exa}
\begin{exa}
For omni-Lie algebra $\pi:\gl(V)\oplus V\to \gl(V)$, the associated Lie triple algebra on $\gl(V)\oplus V$ is given by% and $\gl(V)$ are given by
$$
\begin{aligned}
{[(A, x),(B, y)]}&=(2[A, B], Ay-Bx),\\
[(A, x),(B, y),(C, z)]&=-([[A, B], C],[A, B] z),
%[A,B]&=[A,B]\\
%[A,B, C]&=[[A,B], C].
\end{aligned}
$$
and the action of $\gl(V)$ on $\gl(V)\oplus V$ is given by
$$
\begin{aligned}
{\rho(A)(B, y)}& =([A,B], Ay),\quad \theta(A, B)(C,z)&=0,\\
D(A, B)(C,z)&=-([[A,B],C], [A,B]z).
\end{aligned}
$$
\end{exa}

\begin{exa}
 %Courant algebroid is a vector bundle $E$ equipped with a nondegenerate symmetric bilinear form $\langle$,$\rangle$ , a bilinear bracket $[,]$ on $\Gamma(E)$, and a bundle map $p: E \rightarrow TM$ satisfying the following properties:
%
%(1) The left Jacobi identity $\left[e_1,\left[e_2, e_3\right]\right]=\left[\left[e_1 e_2\right], e_3\right]+\left[e_2,\left[e_1, e_2\right]\right]$,
%
%(2) Anchor is homomorphism $p\left[e_1, e_2\right]=\left[p\left(e_1\right), p\left(e_2\right)\right]$,
%
%(3) Leibniz rule $\left[e_1, f e_2\right]=f\left[e_1, e_2\right]+\mathcal{L}_{p\left(e_1\right)}(f) e_2$,
%
%(4) $[e, e]=\frac{1}{2} \mathcal{D}\langle e, e\rangle$,
%
%(5) Self-adjointness $p\left(e_1\right)\left\langle e_2, e_3\right\rangle=\left\langle\left[e_1, e_2\right], e_3\right\rangle+\left\langle e_2,\left[e_1, e_3\right]\right\rangle$,
%
%where $\mathcal{D}$ is defined as $p^* d: C^{\infty}(M) \stackrel{d}{\rightarrow} \Omega^1(M) \stackrel{p^*}{\rightarrow} E^* \simeq E$.
For the Courant algebroid $\pi:TM\oplus T^*M\to TM$, the associated Lie triple algebra on $\mathcal{X}(M)\oplus \Omega(M)$ is given by
$$
\begin{aligned}
{[(X_1, \xi_1), (X_2, \xi_2)]}&=\left(2[X_1, X_2], \mathcal{L}_{X_1} \xi_2- \mathcal{L}_{X_2} \xi_1-i_{X_2} d\xi_1+i_{X_1} d\xi_2\right)\\
\left\{\left(X_1, \xi_1\right),\left(X_2, \xi_2\right),\left(X_3, \xi_3\right)\right\}&=-\left(\left[\left[X_1, X_2\right], X_3\right], \mathcal{L}_{\left[X_1, X_2\right]} \xi_3-i_{X_3} d\left(\mathcal{L}_{X_1} \xi_2-\mathcal{L}_{X_2} \xi_1\right)\right)
\end{aligned}
$$
and the action of $\mathcal{X}(M)$ on $\mathcal{X}(M)\oplus \Omega(M)$ is given by
$$
\begin{aligned}
{\rho(X_1)(X_2, \xi)}& =([X_1,X_2], \mathcal{L}_{X_1}\xi),\quad \theta(X_1,X_2)(X_3,\xi)&=0,\\
D(X_1,X_2)(X_3,\xi)&=-([[X_1,X_2],X_3], \mathcal{L}_{[X_1,X_2]}\xi).
\end{aligned}
$$
\end{exa}

%\begin{exa}
%For  a Lie triple algebras $\g$, we known that $\g\wedge \g$ is a Leibniz algebra under the operation:
%Let $\ad:\g\wedge \g\to \Inn(\g)$, this is a crossed module of Leibniz algebras.
%Then we get a  crossed module  of Lie triple algebra  $\ad:\g\wedge \g\to \Inn(\g)$  by
%$$
%\begin{aligned}
%{[x\otimes y,u\otimes v]}&=[x,y,u]\ot v+u\ot [x,y, v]\\
%&\quad-[u,v,x]\ot y-x\ot [u,v, y],\\
%{[x\otimes y,u\otimes v, w\otimes t]}&=-[[x,y,u], v,w]\ot t-w\ot [[x,y,u], v, t]\\
%&\quad-[u, [x,y, v],u]\ot t-w\ot [u, [x,y, v],t].
%\end{aligned}
%$$
%and the action of $\Inn(\g)$ on $\g\otimes \g$ is given by
%$$
%\begin{aligned}
%{\rho(\ad(x,y))(u\ot v)}&=\ad(x,y))\trr (u\ot v)=[x,y,u]\ot v+u\ot [x,y, v],\\
%D(\ad(x,y), \ad(u,v))(w\ot t)&=-[\ad(x,y), \ad(u,v)]\trr (w\ot t),\\
%\theta(\ad(x,y), \ad(u,v))(w\ot t)&=0.
%\end{aligned}
%$$
%\end{exa}

It is well known that crossed module extensions of associative algebras and Lie algebras can be classified by the  third cohomology group.
Now we are going to study crossed module extensions of Lie triple algebras and show that they are related to our new cohomology group of Lie triple algebras.

\begin{defi}
 Let $T$ be a Lie triple algebra and   $({V},\rho,\DD,\theta)$ be a representation of $T$. A crossed module extension of  $T$ by $M$ is a short exact sequence
\begin{equation}\label{diagram:exact}
 \xymatrix{
   0  \ar[r]^{} & {M} \ar[r]^{i} & V \ar[r]^{\partial} & S \ar[r]^{\pi} &  T  \ar[r]^{} & 0 \\
  }
\end{equation}
such that $\partial:V\rightarrow S$ is a crossed module and $M\cong \Ker\partial$, $T\cong \mathrm{coKer} \partial$.
\end{defi}

\begin{defi}
 Two crossed module extensions of Lie triple algebras
$\partial:V\rightarrow S$  and ${\partial'}: {V'}\rightarrow {S'}$  are equivalent,
 if there exists a Lie triple algebras homomorphism $\phi:V\to {V'}$ and $\psi: S\to {S'}$  such that the following diagram commutes
 \begin{equation}
\xymatrix{
    0\ar[r]^{} & M  \ar@{=}[d]_{} \ar[r]^{i} & V  \ar[d]_{\phi } \ar[r]^{\partial } & S  \ar[d]_{\psi}  \ar[r]^{\pi } & T \ar@{=}[d]_{} \ar[r]^{} & 0  \\
   0\ar[r]^{} & M  \ar[r]^{i'} &{V}'  \ar[r]^{{\partial}' } & {S}'  \ar[r]^{\pi' } & T  \ar[r]^{} & 0.}
   \end{equation}
The set of equivalent classes of crossed extensions of $T$ by $M$ is denoted by $\mathbf{CExt}( T,{M})$.
\end{defi}

The main result of this section is the following Theorem \ref{mainthm3}.

\begin{thm}\label{mainthm3}
 Let $T$ be a Lie triple algebra and $M$ be a representation of $T$.
Then there is a canonical bijective map:
\begin{equation*}
\Phi:\mathbf{CExt}(T,M)\rightarrow H^{(3,4,4,5)}(T, M).
\end{equation*}
\end{thm}

\begin{proof}
The proof of this  Theorem \ref{mainthm3} will be splitting  into the following steps.

\medskip
Step 1. First,  for a crossed module extension, we define  a $(3,4,4,5)$-cohomology class in the $H^{(3,4,4,5)}(T, M)$ as the follows.
\medskip

Given a crossed module extension of $T$ by ${V}$, we choose linear sections $s : T \rightarrow S$ with $\pi s = \text{id}_V$ and $q : \text{Im}(\partial) \rightarrow V$ with $\partial q = \text{id}$. %We define $\Theta_{\mathcal{E}, s, q}$  a $(3,4,4,5)$-cocycle in the cohomology group of $ {T} $ with coefficients in $M$ as the following Lemma \ref{lem-crossed1} shows.

%\begin{lem}\label{lem-crossed1}
%The maps $\Phi  $ is well-defined.
%\end{lem}

%\begin{proof}
Since $\pi$ is a Lie triple algebra homomorphism, for all $x , y, z\in  {T} ,$ we have
\begin{align*}
\pi ([s(x), s(y)] - s[x,y]) =[\pi s(x), \pi  s(y)] - \pi s[x,y]=[x,y]-[x,y]= 0,
\end{align*}
and
\begin{align*}
\pi ([s(x), s(y),s(z)] - s[x,y,z]) =&[\pi s(x), \pi  s(y), \pi  s(z)] - \pi s[x,y,z]\\
=&[x,y,z]-[x,y,z]= 0.
\end{align*}
This shows that $[s(x), s(y)] - s[x,y]\in\ker(\pi) = \text{Im} (\partial)$ and $[s(x), s(y),s(z)] - s[x,y,z]\in\ker(\pi) = \text{Im} (\partial)$.
Take
$$\nu(x_1,x_2) = q ([s(x_1), s(x_2)] - s[x_1,x_2]) \in V,$$
and
$$\omega(x_1,x_2, x_3) = q ([s(x_1), s(x_2), s(x_3)] - s[x_1,x_2, x_3]) \in V.$$

Define  maps $\Theta_{\mathcal{E}, s, q}=(\theta^3_{\mathcal{E}, s, q},\widehat{\theta}^4_{\mathcal{E}, s, q},\widetilde{\theta}^4_{\mathcal{E}, s, q},\theta^5_{\mathcal{E}, s, q}) $:
$$\theta^3_{\mathcal{E}, s, q}: {T} ^{\otimes 3} \rightarrow M, \quad\widehat{\theta}^4_{\mathcal{E}, s, q}: {T} ^{\otimes 4} \rightarrow M,
\quad\widetilde{\theta}^4_{\mathcal{E}, s, q}: {T} ^{\otimes 4} \rightarrow M,\quad \theta^5_{\mathcal{E}, s, q} :  {T} ^{\otimes 5} \rightarrow M$$
by
\begin{eqnarray*}
&&\theta^3_{\mathcal{E}, s, q} (x_1,x_2, x_3) \\
&=&\omega(x_1,x_2, x_3)+c.p.-[s(x_1),\nu(x_2,x_3)]-c.p.+\nu([x_1,x_2],x_3)+c.p.,\\[1em]
%\end{eqnarray*}
%\begin{eqnarray*}
&&\widehat{\theta}^4_{\mathcal{E}, s, q}( x_1, x_2,x_3, y_1)\\
&=&[s(x_1), \nu(x_2,x_3), s(y_1)]+c.p.-\omega([x_1,x_2], x_3, y_1)-c.p.,\\[1em]
%\end{eqnarray*}
%\begin{eqnarray*}
&&\widetilde{\theta}^4_{\mathcal{E}, s, q}( x_1, x_2,y_1, y_2)\\
&=& \omega( x_1, x_2,[y_1, y_2])-[s(y_1),\omega(x_1,x_2,y_2)]+[s(y_2),\omega(x_1,x_2,y_1)]\\
&&+[s(x_1), s(x_2),\nu(y_1, y_2)]-\nu([x_1, x_2, y_1], y_2) - \nu(y_1,[x_1, x_2,y_2])),
\end{eqnarray*}
and
\begin{eqnarray*}
&&\theta^5_{\mathcal{E}, s, q}(x_1, x_2, y_1, y_2, y_3)\\
&=& \omega(x_1, x_2,[y_1, y_2, y_3])+\DD(x_1, x_2)\omega(y_1, y_2, y_3)\\
&&-\omega([x_1, x_2, y_1], y_2, y_3) - \omega([x_1, x_2,y_2], y_3, y_1) - \omega(y_1, y_2, [x_1, x_2,y_3])\\
&&-\theta(y_2, y_3)\omega(x_1,x_2,y_1)+ \theta(y_1, y_3)\omega(x_1,x_2,y_2) - \DD(y_1, y_2)\omega(x_1,x_2,y_3).
\end{eqnarray*}

Since $\partial$ is a map of crossed modules, it follows that
\begin{eqnarray*}
&&\partial\theta^3_{\mathcal{E}, s, q} (x,y,z)\\
&=&\partial\{\omega(x_1,x_2, x_3)+c.p.-[s(x_1),\nu(x_2,x_3)]-c.p.\\
&&+\nu([x_1,x_2],x_3)+c.p.\}\\
&=& \partial q ([s(x_1), s(x_2), s(x_3)] - s[x_1,x_2, x_3])\\
&&+\{\partial[s(x),   q ([s(y), s(z)] - s[y,z]) ]\\
&&- \partial q ([s([x,y]), s(z)] - s[[x,y],z])\}+c.p.\\
&=&\{[s(x_1), s(x_2), s(x_3)]+[[s(x_1), s(x_2)], s(x_3)]\}\\
&&- \{s[x_1,x_2, x_3]+c.p.+s([[x_1,x_2], x_3])+c.p.\}\\
&&- \{[s(x), s[y,z])]+c.p.+[s([x,y]), s(z)]+c.p.\}\\
&=& - [s(x), s[y,z])]-[s(y), s[z,x])]-[s(z), s[x,y])]\\
&& - [s([x,y]), s(z)]-[s([y,z]), s(x)]-[s([z,x]), s(y)]\\
&=&0.
\end{eqnarray*}
Thus $\theta^3_{\mathcal{E}, s, q} (x,y,z) \in \ker (\partial) = M$. The map $\theta^3_{\mathcal{E}, s, q} :  {T} ^{\otimes 3} \rightarrow M$ is well defined.

We also have
\begin{eqnarray*}
&&\partial\widehat{\theta}^4_{\mathcal{E}, s, q}( x_1, x_2,x_3, y_1)\\
&=&\partial([s(x_1), \nu(x_2,x_3), s(y_1)]+c.p.-\omega([x_1,x_2], x_3, y_1)-c.p.\\
&=&\partial\{([s(x_1), q ([s(x_2), s(x_3)] - s[x_2,x_3]) , s(y_1)]\}+c.p.\\
&&- \partial q\{([s([x_1,x_2]), s(x_3), s(y_1)] - s[[x_1,x_2], x_3,y_1]) -c.p.\}\\
&=&([s(x_1),  ([s(x_2), s(x_3)] , s(y_1)]-([s(x_1),   s[x_2,x_3]) , s(y_1)]+c.p.\\
&&- ([s([x_1,x_2]), s(x_3), s(y_1)] - s[[x_1,x_2], x_3,y_1]) -c.p.\\
&=&0.
\end{eqnarray*}
Thus $\widehat{\theta}^4_{\mathcal{E}, s, q}( x_1, x_2,x_3, y_1)\ \in \ker (\partial) = M$. The map $\widehat{\theta}^4_{\mathcal{E}, s, q}:  {T} ^{\otimes 4} \rightarrow M$ is well defined.

Similarly, one show that $\partial\widetilde{\theta}^4_{\mathcal{E}, s, q}( x_1, x_2,y_1, y_2) =\partial \theta^5_{\mathcal{E}, s, q}(x_1, x_2, y_1, y_2, y_3)=0$, thus the image of these maps are all belong to $\ker (\partial) = M$.
%\end{proof}

Moreover, it is easy to see that the map $\Theta_{\mathcal{E}, s, q}$ is a $(3,4,4,5)$-cocycle ($\delta(\Theta_{\mathcal{E}, s, q})=0$) in the cohomology group of $ {T} $ with coefficients in $M$. Therefore we obtain a map $\Phi$ from $\mathbf{CExt}(T,M)$ to $H^{(3,4,4,5)}(T, M)$.

\medskip
%\begin{lem}
Step 2. We show that $\Theta_{\mathcal{E}, s, q}$ is independent of the choice of section $s$ and $q$.
%\end{lem}
\medskip
%\begin{proof}

We show that the class of $\theta^3_{\mathcal{E}, s, q}$ in $H^{(3,4,4,5)}(T,M)$ does not depend on the section $s$. Suppose $\overline{s} :  {T}  \rightarrow S$ is another section of $\pi$ and let $\theta_{\mathcal{E}, \overline{s}, q}$ be the corresponding $(3,4,4,5)$-cocycle defined by using $\overline{s}$ instead of $s$. Then there exists a linear map $h :  {T}  \rightarrow V$ with $s - \overline{s} = \partial h$.
%Observe that
%\begin{align*}
%&[s(x), g(y,z)] - [\overline{s}(x) , \overline{g} (y,z)]\\
%=~& [s(x) - \overline{s}(x), g (y,z)] + [ \overline{s}(x), (g - \overline{g})(y,z)] \\
%=~& [\partial h (x) , q ([s(y), s(z)] - s[y,z]) ] + [ \overline{s}(x), (g - \overline{g})(y,z)] \\
%=~& - [ ([s(y), s(z)] - s[y,z]), h(x) ] + [ \overline{s}(x), (g - \overline{g})(y,z)]\\
%=~&  [ h(x), [s(y), s(z)] - s[y,z] ] + [ \overline{s}(x), (g - \overline{g})(y,z)].
%\end{align*}
Now we get
\begin{align}\label{e-e'}
&(\theta^3_{\mathcal{E}, s, q} - \theta^3_{\mathcal{E}, \overline{s}, q}) (x,y,z) \nonumber \\
= ~&- [ ([s(y), s(z)] - s[y,z]), h(x) ] + [ ([s(x), s(z)] - s[x,z]), h(y) ] \nonumber \\
~&- [ ([s(x), s(y)] - s[x,y]), h(z) ]  + [ \overline{s}(x), (g - \overline{g})(y,z)]  \nonumber \\
~& - [ \overline{s}(y), (g - \overline{g})(x,z)] + [ \overline{s}(z), (g - \overline{g})(x,y)] \nonumber \\
~~&- (g - \overline{g}) ([x,y],z) + (g - \overline{g}) ([x,z], y) - (g - \overline{g}) ([y,z], x) \nonumber\\
= ~&\{[h(x), ([s(y), s(z)] - s[y,z]),  ]+[ \overline{s}(x), (g - \overline{g})(y,z)]\\
 ~&-(g - \overline{g}) ([x,y],z) \}+c.p. \nonumber
\end{align}

Define a map $b : \wedge^2  {T}  \rightarrow V$ by
\begin{align*}
b (x,y) &= [s(x), h(y)] - h ([x,y]) - [s(y), h(x)] - [ \partial h(x), h (y)]\\
&=[s(x), h(y)]- [s(y), h(x)] - h([x,y]) - [h(x), h(y)].
\end{align*}
Then by direct calculation one shows that  $\partial b = \partial (g - \overline{g})$.
Hence $(g - \overline{g} - b) : \wedge^2  {T}  \rightarrow M=\ker \partial$.
It follows from (\ref{e-e'}) that
\begin{align*}
 &(\theta^3_{\mathcal{E}, s, q} - \theta^3_{\mathcal{E}, \overline{s}, q}) (x,y,z) \\
=~& \{[h(x), ([s(y), s(z)] - s[y,z])] + [ \overline{s}(x), b (y, z)] - b ([x,y],z)\}+c.p.\\
%=~&  [h(x), ([s(y), s(z)] - s[y,z]),  ] + [h(y), ([s(z), s(x)] - s[z,x])] \\
%~& + [h(z), ([s(x), s(y)] - s[x,y]),  ] \\
%~&+ [ \overline{s}(x), b (y, z)] +[ \overline{s}(y), b (z,x)] + [\overline{s}(z) , b (x,y)] \\
%~&- b ([x,y],z) - b  ([z,x], y) - b ([y,z], x)\\
~& + (\delta (g - \overline{g} - b)) (x,y,z)\\
=~& (\delta (g - \overline{g} - b)) (x,y,z).
\end{align*}
%In the right hand side of the above equation, we substitute the definition of $b$ in last six terms.
%After many cancellations on the right hand side, we have
%\begin{align*}
%~& \{[h(x), ([s(y), s(z)] - s[y,z])] + [ \overline{s}(x), b (y, z)] - b ([x,y],z)\}+c.p.\\
%=~&\{[h(x), ([s(y), s(z)] - s[y,z])]\\
%~&+ [(s-\partial h)(x), b (y, z)] - b ([x,y],z)\}+c.p.\\
%=~&\{[h(x), [s(y), s(z)]] - [h(x),s[y,z]]\\
%~&+ [s(x),b(y, z)]-[h(x), b(y, z)] - b ([x,y],z)\}+c.p.\\
%=~&\{[h(x), [s(y), s(z)]] - [h(x),s[y,z]]\\
%~&+ [s(x),[s(y), h(z)]- [s(x),[s(z), h(y)]] \\
%~&- [s(x),h([y,z])]  - [s(x),[h(y), h(z)]]\\
%~&-[h(x),[s(y), h(z)]]+ [h(x),[s(z), h(y)]] \\
%~&+ [h(x), h([y,z])]  + [h(x),[h(y), h(z)]]\\
%~&-[s([x,y]), h(z)]+ [s(z), h([x,y])] \\
%~&+ h([[x,y],z]) + [h([x,y]), h(z)]\}+c.p.\\
%=~&0.
%\end{align*}
%Thus we are only left with the term $(\delta (g - \overline{g} - b)) (x,y,z)$.
Thus the class of $\theta^3_{\mathcal{E}, s, q}$ does not depend on the section $s$.
Similarly, one show that $\partial\widetilde{\theta}^4_{\mathcal{E}, s, q}$, $\partial\widehat{\theta}^4_{\mathcal{E}, s, q}$, $\partial \theta^5_{\mathcal{E}, s, q}$ are independent on the section $s$.
%\end{proof}

%\begin{lem}\label{thm:2-cocylce}
Moreover, one show that two crossed module extensions of Lie triple algebras are equivalent if and only if $\Theta_{\mathcal{E}, s, q}$ and $\Theta'_{\mathcal{E}, s, q}$ are in the same cohomology class.

\medskip
Step 3.  We define a map $\Psi:H^{(3,4,4,5)}({T}, {M}) \rightarrow{\bf CExt}({T}, {M})$ as follows.
\medskip

Let $\mathcal{F}({T})$ be the free Lie triple algebra of ${T}$ over the vector space ${T}$ using the similar method in \cite{MS,St}\footnote{In fact, we only use the exists of free Lie triple algebra but not its precisely construction in this paper.} and ${\pi}: \mathcal{F}({T}) \rightarrow {T}$ be the canonical projection. Consider the cochain complexes
$$
\begin{aligned}
&{C}^{(2,3)}({T}, {M}) \quad\stackrel{\Delta_2}{\longrightarrow}\quad{C}^{(3,4,4,5)}({T}, {M}) \stackrel{\Delta_3}{\longrightarrow} \cdots \\
& \qquad \pi^* \downarrow \quad\quad\quad\quad\quad\qquad  \pi^* \downarrow \\
&{C}^{(2,3)}(\mathcal{F}({T}), {M}) \stackrel{\Delta_2}{\longrightarrow}{C}^{(3,4,4,5)}(\mathcal{F}({T}), {M}) \stackrel{\Delta_3}{\longrightarrow} \cdots \\
\end{aligned}
$$
%$$
%\begin{aligned}
%& \quad{C}^1({T}, {M}) \quad\stackrel{\partial^1}{\longrightarrow}\quad{C}^2({T}, {M}) \quad\stackrel{\partial^2}{\longrightarrow}\quad{C}^3({T}, {M}) \stackrel{\partial^3}{\longrightarrow} \cdots \\
%&\quad\quad\quad v^1 \downarrow \quad \quad\quad\quad\quad v^2 \downarrow \quad\quad\quad\quad\quad\quad v^3\downarrow \\
%& {C}^1(\mathcal{F}({T}), {M}) \stackrel{\partial^1}{\longrightarrow}{C}^2(\mathcal{F}({T}), {M}) \stackrel{\partial^2}{\longrightarrow}{C}^3(\mathcal{F}({T}), {M}) \stackrel{\partial^3}{\longrightarrow} \cdots \\
%\end{aligned}
%$$
 For any (3,4,4,5)-cocycle $(l_3,\widehat{l}_4, \widetilde{l}_4,l_5)\in {C}^{(3,4,4,5)}({T}, {M})$, then there exists a cochain $(\nu,\omega)\in {C}^{(2,3)}(\mathcal{F}({T}), {M})$ such that $ \pi^*  (l_3,\widehat{l}_4, \widetilde{l}_4,l_5)=\Delta_2(\nu,\omega)$. Now we claim that the following sequence is a crossed extension of ${T}$ by ${M}$
$$
\mathscr{E}: 0 \longrightarrow {M} \stackrel{i}{\longrightarrow} {M} \oplus {V} \stackrel{\mu }{\longrightarrow} \mathcal{F}({T}) \stackrel{\pi}{\longrightarrow} {T} \longrightarrow 0,
$$
where ${V}=\operatorname{Ker} \pi$. The Lie triple algebra structure on ${M} \oplus {V}$ is given by
\begin{eqnarray*}
\left[m_1+ v_1,m_2+v_2\right]&=&\nu\left(v_1, v_2\right)+[v_1, v_2]\\
\left[[m_1+ v_1,m_2+v_2, m_3+v_3\right]&=&\omega\left(v_1, v_2, v_3\right)+[v_1, v_2, v_3],
\end{eqnarray*}
the action of $\mathcal{F}({T})$ on ${M} \oplus {V}$ is given by
%\begin{eqnarray*}
%{\rho(x_1)(m_1+v_1)}_\nu&\triangleq&[x_1, x_2] + \nu(x_1, x_2)+\rho(x_1)(u_2) -\rho(x_2)(u_1),\\
%\notag{[x_1 + u_1, x_2 + u_2, x_3 + u_3]}_\omega&\triangleq&[x_1, x_2, x_3] + \omega(x_1, x_2, x_3)+ \DD(x_1, x_2)(u_3)\\
%&&-\theta(x_1, x_3)(u_2) + \theta( x_2, x_3)(u_1),
%\end{eqnarray*}
%$$
%\begin{aligned}
%{[x+u, y+v]}& =[x, y]_{{T}}+\rho(x) (v)-\rho(y) (u)+[u, v]_{{V}}, \\
%{[x+u, y+v, z+w ]} & =[x, y, z]_{{T}}+D(x, y) (w)+\theta(y, z) (u)-\theta(x, z) (v)+[u, v, w]_{{V}},
%\end{aligned}
%$$
\begin{eqnarray*}
&&{\rho(x)(m+v)}= \rho(\widetilde{x})(m)+\nu(x, v)+[x,v],\\
&& \theta({x_1}, {x_2})(m+v)=\theta(\widetilde{x}_1, \widetilde{x}_2)(m)+\omega({x_1}, {x_2}, v)+[{x_1}, {x_2}, v], \\
%&& {[{x_1},(m, v), {x_2}]=([v {x_1}, m, v {x_2}]+\omega({x_1}, v, {x_2}),[{x_1}, v, {x_2}]),} \\
&& \DD({x_1}, {x_2})(m+v)=\DD(\widetilde{x}_1, \widetilde{x}_2))(m)+\omega(v, {x_1}, {x_2})+[v, {x_1}, {x_2}].
%&& {\left[{x},\left(m_1, v_1\right),\left(m_2, v_2\right)\right]=\left(\omega\left({x}, v_1, v_2\right),\left[{x}, v_1, v_2\right]\right),} \\
%&& {\left[\left(m_1, v_1\right), {x},\left(m_2, v_2\right)\right]=\left(\omega\left(v_1, {x}, v_2\right),\left[v_1, {x}, v_2\right]\right),} \\
%&& {\left[\left(m_1, v_1\right),\left(m_2, v_2\right), {x}\right]=\left(\omega\left(v_1, v_2, {x}\right),\left[v_1, v_2, {x}\right]\right)}
\end{eqnarray*}
and the map $\mu $ is defined by $\mu (m, v)=v$, for every $m, m_i \in {M}, v, v_i \in {V}$,  ${x}, {x_1}, {x_2} \in \mathcal{F}({T})$
and $\widetilde{x}_1=\pi({x_1}),\widetilde{x}_2=\pi({x_2})\in {T}$.
It is routine to show that this defines an action of $\mathcal{F}({T})$ on ${M} \oplus {V}$.
Thus we obtain that $\mathscr{E}$ is a crossed extension of ${T}$ by ${M}$ and we define $\Psi (l_3,\widehat{l}_4, \widetilde{l}_4,l_5)$ to be the class of $\mathscr{E}$ in ${\bf CExt}({T}, {M})$.

Finally, it is straightforward to show that that $\Psi$ is a well defined map from $H^{(3,4,4,5)}({T}, {M})$ to ${\bf CExt}({T}, {M})$,
%that is the class of $\mathscr{E}$ does not depend on the cohomology class of $h$ in $H^{(3,4,4,5)}({T}, {M})$.
and the maps $\Phi$ and $\Psi$ are inverses to each other.
Thus by the above argument, we obtain $\Phi$ is a bijective map.
The proof is completed.
\end{proof}

\section*{Acknowledgements}
We would like to thank Jonatan Stava for  answering our questions about free invariant connection algebras and free Lie triple algebras in \cite{MS,St}.

\section*{Declarations}

 {\bf Competing interests} There are no conflicts of interest for this work.

\noindent {\bf Availability of data and materials} Data sharing is not applicable to this article as no datasets were generated or analyzed during the current study.

\vskip7pt
\footnotesize{
\noindent Tao Zhang\\
College of Mathematics and Information Science,\\
Henan Normal University, Xinxiang 453007, P. R. China;\\
 E-mail address: \texttt{{zhangtao@htu.edu.cn}}

\vskip7pt
\footnotesize{
\noindent Zhang-Ju Liu\\
College of Mathematics and Statistics,\\
Henan University, Kaifeng 475000, P. R. China;\\
 E-mail address:\texttt{{ liuzj@pku.edu.cn}}
}

\end{document}